\documentclass[12pt]{article}
\usepackage{latexsym, amssymb, amsthm, amsmath}
\usepackage{dsfont}
\usepackage{graphicx}
\usepackage{color}

\DeclareMathAlphabet\gothic{U}{euf}{m}{n}

\makeatletter
\def\eqnarray{\stepcounter{equation}\let\@currentlabel=\theequation
\global\@eqnswtrue
\tabskip\@centering\let\\=\@eqncr
$$\halign to \displaywidth\bgroup\hfil\global\@eqcnt\z@
  $\displaystyle\tabskip\z@{##}$&\global\@eqcnt\@ne
  \hfil$\displaystyle{{}##{}}$\hfil
  &\global\@eqcnt\tw@ $\displaystyle{##}$\hfil
  \tabskip\@centering&\llap{##}\tabskip\z@\cr}

\def\endeqnarray{\@@eqncr\egroup
      \global\advance\c@equation\m@ne$$\global\@ignoretrue}

\def\@yeqncr{\@ifnextchar [{\@xeqncr}{\@xeqncr[5pt]}}
\makeatother

\allowdisplaybreaks[3]

\newtheorem{theorem}{Theorem}[section]
\newtheorem{thm}[theorem]{Theorem}

\newtheorem{lem}[theorem]{Lemma}

\newtheorem{cor}[theorem]{Corollary}

\newtheorem{prop}[theorem]{Proposition}

\theoremstyle{definition}

\newtheorem{defn}[theorem]{Definition}
\newtheorem{assu}[theorem]{Assumption}

\newtheorem{rem}[theorem]{Remark}

\theoremstyle{remark}

\newcounter{teller}

\newcounter{tellerr}

\newcounter{tellerrr}

\newcounter{proofstep}

\newcommand{\dual}[2]{\langle{#1}\,|\,{#2}\rangle}
 \newcommand{\upp}[1]{{#1}^\bullet}
 \newcommand{\low}[1]{{#1}_\bullet}
\newcommand{\with}{\,:\,}
 
 \newcommand{\der}{\,\mathrm{d}}

\newcommand{\field}[1]{\mathbb{#1}}

\newcommand{\R}{\field{R}}

 \newcommand{\lrdual}[2]{\left\langle{#1}\,|\,{#2}\right\rangle}

\newcommand{\supp}{\mathop{\rm supp}}

 \newcommand{\Diri}{{D}}

 \newcommand{\Neumann}{\Gamma}

 \newcommand{\reaction}{r}
 \newcommand{\chempot}{\chi}
 
 \newcommand{\ie}{i.e.}
 \newcommand{\eg}{e.g.}

\begin{document}
\bibliographystyle{amsplain}

{\bf The 3D transient semiconductor equations with gradient-dependent
 and interfacial recombination}

\begin{center}
Karoline Disser and Joachim Rehberg
\end{center}

\begin{abstract} 
We establish the well-posedness of the transient van Roosbroeck 
system in three space dimensions under realistic assumptions on the data: non-smooth domains, discontinuous
coefficient functions and mixed boundary conditions. Moreover, within this analysis,
 recombination terms may be concentrated on surfaces and interfaces and may not only depend on charge-carrier densities, but also on the 
electric field and currents. In particular, this
includes Avalanche recombination. The proofs are based on recent abstract results on maximal parabolic and
optimal elliptic regularity of divergence-form operators.  
\end{abstract}

\par
\emph{Acknowledgements.} The authors would like to thank Herbert Gajewski, Annegret Glitzky, Thomas Koprucki, Matthias Liero, Hagen Neidhardt and Marita Thomas for stimulating discussions and an ongoing exchange of ideas on this topic. K.D. was partially supported by the European Research 
Council via ``ERC-2010-AdG no. 267802 (Analysis of Multiscale Systems 
Driven by Functionals)'' and by the DFG International Research Training Group IRTG
1529. 

\tableofcontents

\vspace*{10mm}

\newpage

\section[Introduction]{Introduction}
\label{sec:intro}

In 1950, van~Roosbroeck \cite{roosbroeck50} established a system of
partial differential equations describing the dynamics of electron and
hole densities in a semiconductor device due to drift and diffusion within a
self-consistent electrical field.
In 1964, Gummel \cite{gummel64}
published the first report on the numerical solution of these
drift--diffusion equations for an operating semiconductor device. 
In the mathematical literature, there are now a number of related models and results. For excellent overviews, see \cite{jerome} or \cite{Mark} and references therein. Very active recent areas of research are, for example, the modelling and analysis of hydrodynamic models, active interfaces, e.g. in solar cells, and organic semiconductors, \cite{Glitzky, GL17, Gran, GS, JZ, HsiaoMarkWang, WuMarkZheng}.    
In real device simulation, drift-diffusion formulations and adaptive codes based on van Roosbroeck's system represent the state of the art, \cite{FKF, GLNS, sentaurus}. Regarding the numerics and analysis of these systems, we highlight three main difficulties:
\begin{itemize}
	\item The devices exhibit \emph{non-smoothness}, referring to non-smooth boundary regularity of their domains, inhomogeneous, mixed boundary conditions due to external contacts, and discountinuous material coefficients due to their heterogeneous, mostly layered, structure.
	\item The dynamics include \emph{nonlinearities} of high order, both in the expressions for the currents and for recombination, depending, for example, on the electric field itself rather than its potential. A highly relevant prototype is \emph{Avalanche recombination}.
	\item Some processes concentrate on or are active on lower-dimensional substructures only, like surfacial or \emph{interfacial recombination} due to material structure or impurities.  
	\end{itemize}
	The aim of this paper is to establish  a functional analytic setting for van Roosbroeck's system that allows us to simultaneously handle these aspects. It is tayolered exactly to the combination of a lack of regularity due to non-smoothness, and the need for regularity due to nonlinearity (we refer to a more detailed discussion in Section \ref{sec:nl-reform}). In particular, even though interfacial recombination in general prevents the existence of strong solutions, we can show well-posedness in a suitable norm and H\"older regularity of solutions, cf. Theorem \ref{mr}.
These results provide a strong basis for further numerical analysis, cf. for example the discussion in \ref{sec:nl-poisson}, for the modeling of more complex devices and coupled effects, and for future optimization and optimal control of the system.\\ 
The first proof of global existence and uniqueness of
weak solutions for van Roosbroeck's system \emph{under realistic physical and geometrical conditions}
is due to Gajewski and Gr{\"o}ger  
\cite{gajewski:groeger86,gajewski:groeger89}. It was shown that the solution
tends to thermodynamical equilibrium, if this is admitted by the boundary conditions. 
The key for proving these results is a Lyapunov functional.
At least one serious drawback of these and related results is that only recombination terms are admissable which depend on the densities, and this mostly even under some
additional structural conditions, see 
\cite[2.2.3]{gajewski93}, \cite[Ch.~6]{gajewski:groeger90},  \cite{skrypnik:02} and \cite{veiga}. The only exception seems to be the paper of Seidman \cite{Seid}, where
Avalanche generation -- also called \emph{impact ionization}, is included. However,  his analytic framework requires (generically) smooth geometries and necessarily excludes mixed boundary conditions, cf. \cite[Ch. 5]{Seid}, and interfacial recombination, which are essentially indispensable for real device modeling.\\
On the other hand, Avalanche generation is the determining operating priniciple of both Avalanche diodes and 
Avalanche transistors, \cite{ebers, hamilt, spirit}, and it 
is of interest for modeling solar cells, see \cite{kolod, marti}. 
In the case of Avalanche generation, no energy functional for van Roosbroeck's system is known and, as is already observed in 
\cite{Seid}, methods based on maximum principles are not applicable. 
Thus, global existence cannot be expected (and may not be desirable) in such a general context, compare \cite{fila, kenned}, \cite[p. 55]{Marko}.\\
Hence, our approach is different and rests on a reformulation of the system as the nonlocal quasilinear dynamics of the quasi Fermi levels, in an appropriate Banach space, cf. Section \ref{sec:nl-reform} and cf. \cite{vanroos2d} for a similar approach to the two-dimensional problem in an $L^p$-space without Avalanche recombination. We can then show well-posedness using maximal parabolic regularity of the linearized problem and the contraction mapping principle. Some special (mathematical) aspects of this approach are the following:
\begin{itemize}
	\item It includes a detailed analysis of the nonlinear Poisson equation specific to the system. This also gives rise to efficient numerical schemes, compare \cite{gajewski93} and the discussion in Subsection \ref{sec:nl-poisson}. 
	\item A quite elaborate choice of the underlying Banach space, providing the spatial regularity of rates a.e. in time, cf. Section \ref{sec:main}. In particular, spaces of types $L^p$ and $W^{-1,2}$ are excluded by non-smoothness, interfacial terms and nonlinearity, respectively, and spaces of type $W^{-1,p}$ are also not suitable. Our choice can be viewed as an adequate framework for the treatment
of generalized second-order quasilinear parabolic problems with nonsmooth data when including semilinear terms that depend on (powers of) gradients
of the unknowns.
	\item Many intricate properties of the non-smooth Poisson operators $-\mathrm{div} \, \mu \nabla \cdot $, entering in the equation for the electrostatic potential and the current fluxes,  are essential to the analysis and were achieved only recently (see e.g.
Proposition \ref{t-zwischen}  and references):
\begin{itemize}
	\item They provide topological isomorphims between the spaces $W^{1,q}_D(\Omega) $ and $W^{-1,q}_D(\Omega) $
with $q$ larger than the space dimension $3$, cf. Assumption \ref{assu:iso}. 
An assumption like this was already introduced in \cite{gajewski:groeger89} (compare \cite[Introduction]{veiga})
 as an ad hoc
 assumption in order to show uniqueness in case of Fermi-Dirac statistics,
 but is now substantially covered by \cite{disser}  in cases
of mixed boundary conditions and heterogeneous, \emph{layered} materials. Here, `layered' can be interpreted in a fairly broad sense that may cover many specific devices. 
\item They have maximal parabolic
regularity, even when considered on interpolation spaces of $W^{-1,q}$ and $L^q$, cf. Proposition \ref{t-zwischen}.
\item Even with varying coefficients due to the quasilinearity of the system, they have a (sufficiently regular) common domain of definition on these interpolation spaces, and the operator norm can be estimated suitably, cf. Lemma \ref{t-multipl}.
\item The domains of (suitable) fractional powers can be determined, due to the pioneering results of \cite{ausch}. In particular, it can be shown that they may embed into $W^{1,q}$.
\end{itemize} 
	\end{itemize}
Even with some technicalities in the functional analytic framework, we want to present a main result that is straightforwardly applicable to real devices. Thus, we have taken care to motivate and discuss the mathematical assumptions, using known results, examples, relevant physical quantities and additional figures. \\ 
The outline of this paper is as follows: In the next section, we introduce van Roosbroeck's model, including examples of expressions for bulk and surface recombination. In Section \ref{sec:rigorous}, we collect mathematical prerequisites. In particular, this includes assumptions and preliminary results associated to the non-smoothness of the setting and inhomogeneous data and to Avalanche recombination. In Section \ref{sec:nl-reform}, we  introduce and explain the functional analytic setting, analyse the nonlinear Poisson equation for the electrostatic potential given in terms of quasi Fermi levels, and deduce how the system can then be rewritten as a quasilinear abstract Cauchy problem. In Section \ref{sec:main}, we prove the main result on well-posedness and discuss regularity of solutions. 
%
\section[The van Roosbroeck system]{The van Roosbroeck system}
\label{sec:setting}
%
 In this section we introduce the van Roosbroeck system for modeling the transport of charges in semiconductor devices. 
Therein, the negative and positive charge carriers, electrons and holes, 
move by diffusion and drift in a self-consistent electrical field and on their 
way, due to various mechanisms, they may recombine to charge-neutral 
electron-hole pairs or, vice versa, negative and positive charge carriers 
may be generated from charge-neutral electron-hole pairs.  
\par 
The electronic state of the semiconductor device resulting from these phenomena  
is described by the triple $(u_1,u_2,\varphi)$ of {\bf unknowns} that consists of 
\begin{itemize}
\item 
the densities $u_1$ and $u_2$ of electrons and holes, and
\item 
the electrostatic potential $\varphi$.  
\end{itemize}
Moreover, further physical quantities associated with $(u_1,u_2,\varphi)$ are used to describe the state of the device: 
\begin{itemize}
\item 
the chemical potentials $\chi_1$ and $\chi_2$, 
\item
the quasi Fermi levels
$\Phi_1, \Phi_2$, and,
\item 
the electron and hole currents $j_1$ and $j_2$. 
\end{itemize} 
Their precise relations are given in Section \ref{relations}.    
\par 
Throughout this work we assume that the semiconductor device occupies a bounded domain $\Omega\subset\R^3$.  
Its boundary $\partial\Omega$ with outer unit normal $\nu,$ 
consists of a Dirichlet part $D\subset \partial\Omega$ and of a Neumann, resp.\ Robin part 
$\Gamma:=\partial\Omega\backslash D$. In addition, two-dimensional interfaces $\Pi \subset \Omega$ 
are taken into account, where additional recombination mechanisms may take place, triggered e.g.\ by material impurities. 
The precise mathematical assumptions on the geometry of these objects are collected in Assumption \ref{spatial}. 
The evolution of the charge carriers is monitored during a finite 
time interval $J=]0,T[$ with $T \in ]0, \infty[$. 
\par 
The {\bf van Roosbroeck system} \eqref{vanRoos}, defined on $J\times\Omega$, 
then consists of the {\bf Poisson equation \eqref{Poisson-eq}} and the {\bf current continuity equations \eqref{CuCo-eq}}:   
\begin{subequations}
\label{vanRoos}
\begin{equation}
\label{Poisson-eq}
\begin{aligned}[2]
\text{Poisson equation:}\hspace*{3cm}     -\mathrm{div} 
    \left( \varepsilon \nabla \varphi
    \right) 
    & =  {d} + u_1 - u_2 &
    \quad&
    \text{in} \, \Omega, \\
\varphi & = {\varphi}_D 
    &\quad& 
    \text{on }D, \\ 
   {\nu}\cdot{ 
      \left(
        \varepsilon \nabla \varphi
      \right)
    }
    + 
   \varepsilon_{\Gamma} \varphi 
   & ={\varphi}_\Gamma 
    &\quad&
    \text{on }\Gamma,
\end{aligned} 
\end{equation}
and with $k\in\{1,2\},$ $k=1$ for electrons and $k=2$ for holes, the 
\begin{equation}
\label{CuCo-eq}
\begin{aligned}[2]
\text{current-continuity equation:}\qquad
  \partial_t u_k - \text{div} j_k 
  &= r^\Omega  &\quad& \text{in } J\times (\Omega \setminus \Pi)\\
    \Phi_k(t) &=\Phi_k^D(t) 
    &\quad 
    &\text{on $\Diri$,} 
    \\   
{\nu}\cdot{j_k} 
    &= {r}^\Gamma 
    &\quad 
    &\text{on $\Neumann$,} 
\\
    [{\nu}\cdot{j_k}] 
    &= {r}^ \Pi 
    &\quad 
    &\text{on $ \Pi$.} 
\end{aligned}
\end{equation}
 The evolution starts from initial conditions $\Phi_k(0)=\Phi_{k,0}$.
\end{subequations}\\
The parameters in the Poisson equation are the dielectric permittivity $\varepsilon:\Omega\to\R^{3\times 3}$ and, 
on the right-hand side, the (prescribed) doping profile $d$. The latter is allowed to be located also on 
a two-dimensional surface in $\overline \Omega$ (cf. \cite{naz} \cite{drumm}), see our mathematical requirement on $d$ in Assumption 
\ref{assu:poi-diri-bv} below.
Moreover, in the corresponding boundary conditions,  
$\varepsilon_{\Gamma}:\Gamma\to[0,\infty)$ represents the capacity of the part of the corresponding
device surface, ${\varphi}_D$ and ${\varphi}_\Gamma$ are the voltages applied at 
the contacts of the device, and may, therefore depend on time. \\
From now on we denote the pair $(\Phi_1,\Phi_2)$ of quasi Fermi levels by $\Phi$.
Analogously, we always write $u$ for the pair of densities $(u_1,u_2)$.
\par 
The current-continuity equations feature the currents $j_k$ on their left-hand side and reaction or recombination terms 
$r^\Omega, r^\Gamma, r^\Pi$ on their right-hand side. Here, $r^\Omega$ acts in the bulk and, additionally, the Neumann conditions
in \eqref{CuCo-eq} balance the normal fluxes cross the exterior boundary $\Gamma$ with surface recombinations $r^\Gamma$
taking place on $\Gamma,$ resp.\ the jump of the normal fluxes $ [{\nu}\cdot{j_k}] $ across $\Pi$ with surface recombinations
${r}^ \Pi$ taking place on the surface $\Pi$. Details on $j_k$ and $r^\Omega$,  $r^\Gamma$,  $r^\Pi$ and in particular on their
dependence of the quantities $u$,$\varphi$, and $\Phi$ are given in Sections \ref{relations} and \ref{sec:reactions}. 

%
%

%
\subsection{Carrier densities and currents}
\label{relations}
%
An essential modeling ingredient of van Roosbroeck's system is the relation of 
 the densities of electrons and holes with their chemical potentials. We assume
\begin{equation} \label{eq:densities}
  u_k(t,x) 
 =
  \mathcal{F}_k 
  \left(
    \chempot_k(t,x)
  \right)\, ,
   \quad  x \in \Omega 
  ,
  \qquad k=1,2,
\end{equation}

where the functions $\mathcal{F}_1$ and $\mathcal{F}_2$ represent the
statistical distribution of the electrons and holes in the energy band.  In general, Fermi--Dirac statistics applies, \ie\ 
\begin{equation}
  \label{eq:fermi-integral}
\mathcal{F}_k (s) 
  =
  \frac{2}{\sqrt{\pi}}  
  \int_0^\infty 
  \frac{\sqrt{t}}{1+e^(t-s)}
  \der{t}
  ,
  \qquad
  s\in \R
  ,
  \qquad k=1,2
  .
\end{equation}
Sometimes, Boltzmann statistics is a good approximation: 
\begin{equation}
  \label{eq:boltzmann}
  \mathcal{F}_k (s) 
  = e^s.
\end{equation}

As is common, we assume that the electron and hole current is driven by
the gradient of the quasi Fermi level of electrons $\Phi_1$
and holes $\Phi_2$, respectively. More precisely, the
currents are given by
\begin{equation}
  \label{eq:curr-dens}
  j_k(t,x) =  u_k(t,x) \mu_k(x) \, \nabla \Phi_k(t,x)\; , 
  \quad x \in \Omega
  ,
  \qquad k = 1,2, 
\end{equation}
where the  quasi Fermi levels $\Phi_k$ are related to the chemical potentials $\chi_k$ via
\begin{equation} \label{e-relation}
\chi_k=\Phi_k+(-1)^k\varphi, \quad k=1,2 
\end{equation}
Here, $\mu_k:\Omega\to\R^{3\times 3}$ are the mobility tensors for electrons and holes, respectively.
 We specify the mathematical prerequisites on the functions $\mathcal F_k$ in the following 
\begin{assu}
  \label{assu:distri}
 The functions $\mathcal F_k \colon
      \mathbb{R}
      \to 
      ]0,\infty[$, $k=1,2$
are twice continuously differentiable with $\mathcal F_k(s) \to +\infty$ as $s \to +\infty$. 
Moreover, their derivatives $\mathcal F_k'$ are bounded from above and below on bounded intervals by strictly positive constants.
\end{assu}
This includes Boltzmann statistics \eqref{eq:boltzmann}, as well as Fermi--Dirac statistics
  \eqref{eq:fermi-integral}, for the distribution functions. 
%
\subsection[Recombination terms]{Recombination terms}
\label{sec:reactions}
%

The recombination term $r^\Omega$ on the right-hand side of the current--continuity
equations \eqref{CuCo-eq} can be given by rather general functions of the electrostatic potential, of the 
currents, and of the vector of electron/hole densities. It describes the production
of electrons and holes, respectively --- production or destruction, depending on the sign. Our formulation of the reaction
rates remains abstract, cf. Section \ref{sec:rigorous}, but in particular, it includes a variety of models 
 for semiconductors. It covers non-radiative
recombination like the Shockley--Read--Hall recombination due to phonon 
transition and Auger recombination (three particle transition) as well as Avalanche generation,
see \eg\ \cite{selberherr84,landsberg91,gajewski93} and the references
cited there.

\subsubsection{Bulk recombination}
A rather general model for many recombination terms, valid under \emph{any} statistics, is 
\begin{equation*} 
r^\Omega (u_1, u_2, \Phi_1, \Phi_2) = \hat r(u_1,u_2)\bigl (g - \exp (\Phi_1+\Phi_2)\bigr ),
\end{equation*}
cf. \cite[Sect. 9.2]{bonc}. 
In case of \emph{Boltzmann statistics}, this includes the well-known Shockley--Read--Hall recombination (SRH)
and the Auger recombination (AUG):\\
(SRH) \emph{Shockley--Read--Hall recombination} :                             %
\begin{equation}\label{e-shockleyrh}
\reaction^\Omega_{SRH} 
  =
  \frac{u_1 u_2 - n_i^2}{\tau_2(u_1+n_1)+\tau_1(u_2+n_2)},
  \qquad
\end{equation}
where $n_i$ is the intrinsic carrier density, $n_1$, $n_2$ are
reference densities, and $\tau_1$, $\tau_2$ are the lifetimes of
electrons and holes, respectively.  $n_i$, $n_1$, $n_2$, and $\tau_1$,
$\tau_2$ are parameters of the semiconductor material.\\
(AUG) 
\emph{Auger recombination} (three particle transitions):
\begin{equation} \label{e-auger}
\reaction^\Omega_{Auger} 
=
  (u_1 u_2 - n_i^2)(c_1^{Auger} u_1 + c_2^{Auger}u_2),
  \qquad
\end{equation}
where $c_1^{Auger}$ and $c_2^{Auger}$ are the Auger capture
coefficients of electrons and holes, respectively, in the
semiconductor material.\\
(AVA) An analytical expression for \emph{Avalanche generation} (impact ionization), valid at least in the cases of Silicon or Germanium, is 
\begin{equation} \label{e-0}
 r^\Omega_{Ava}(u,\varphi,\Phi)=c_n |j_n| \exp \Bigl ({\frac {-a_n}{|E \cdot \mathfrak j_{n}|}}\Bigr )
+c_p |j_p| \exp \Bigl ({\frac {-a_p}{|E \cdot \mathfrak j_{p}|}}\Bigr ),
\end{equation}
where $E=\nabla \varphi$ is the electrical field and $\mathfrak  j_{n,p}$ are the normalized currents
$\frac {j_{n,p}}{|j_{n,p}|}$ of the corresponding type. The parameters $a, c_{n,p}$
are given, see \cite[p. 111/112]{selberherr84} and references; in particular 
Tables 4.2-3/4.2-4, and see also  \cite[Ch. p. 17, p. 54/55]{Marko}.

\subsubsection{Surface recombination}
Our model also allows for surface recombination terms $r^\Gamma$ along 
an exterior (Neumann/Robin) part of the boundary and  $r^\Pi$ along interior, 2-dimensional surfaces $\Pi$,
cf. \cite[p. 110]{selberherr84} and references given there, see also \cite{skrypnik:02}.
Of course,  if $r^\Gamma \equiv 0$, then the semiconductor is isolated at $\Gamma$, i.e the current through
$\Gamma$ is zero.\\
The functional analytic requirements on the reaction terms are specified in
Subsection \ref{Reco}. A typical example of  surface recombination is  analogous to Shockley-Read-Hall, at gate contacts, 
\begin{equation*}
	r^\Gamma_{Surf}(u) =  \frac{u_1 u_2 - n_i^2}{v_2(u_1+n_1)+v_1(u_2+n_2)},
	\end{equation*} 
with additional parameters $v_1, v_2$.

%
\section[Prerequisites]{Mathematical prerequisites and assumptions}
\label{sec:rigorous}
%
In this section, we introduce some mathematical terminology and state mathematical prerequisites for the analysis of the van Roosbroeck system \eqref{vanRoos}.
\par 
In particular, we have the following requirements on the domain $\Omega$ occupied by the device.  Figure \ref{fig-technoa} shows a typical example. 

\begin{figure}[htbp]
\centerline{
\includegraphics[height=2in]
{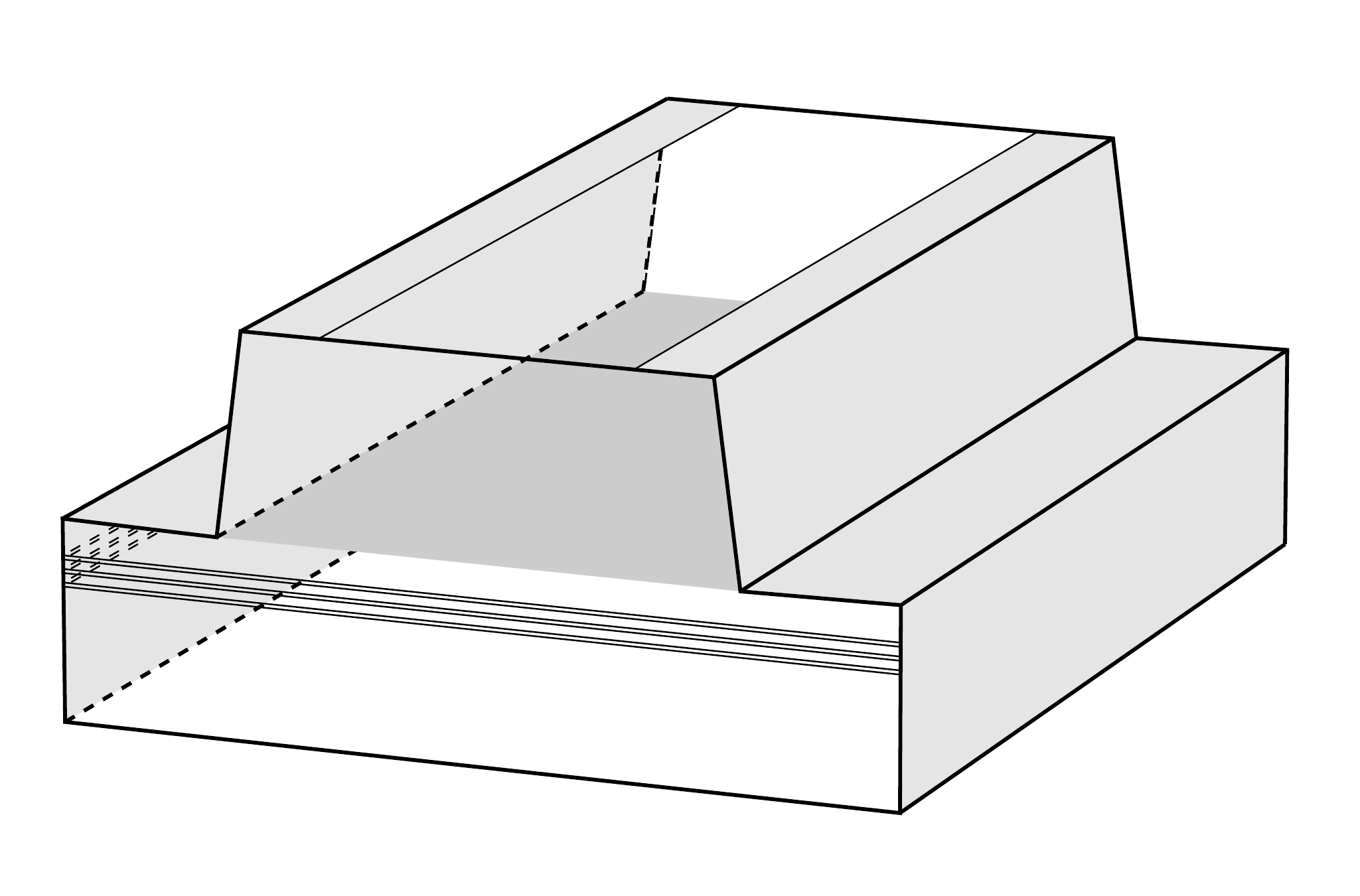}
}
\caption{\label{fig-technoa} Scheme of a ridge waveguide quantum well laser 
(detail 3.2$\mu m \times 1.5 \mu m \times 4\mu m$). The device domain has two material layers. The material
interface (darkly shaded) and the Neumann boundary part (lightly shaded) meet at an edge. At the bottom and the top of the structure, contacts give rise
to Dirichlet boundary conditions for the electrostatic potential. A triple quantum well
structure is indicated where the light beam forms in the symmetry plane of the domain.
}
\end{figure}

\begin{assu} \label{spatial}
The device under consideration occupies a bounded domain $\Omega\subset\R^3$. The
 boundary $\partial \Omega$ is decomposed into a Dirichlet boundary part $D$ and
 its complement $\Gamma:=\partial \Omega \setminus D$.
It holds that
\begin{itemize}
\item the Dirichlet boundary part $D$ is a $(d-1)$-set in the sense of Jonsson/Wallin, cf. \cite[Ch. II]{jons}), and that 
\item every point $x$ in the closure of $\overline \Gamma$ admits a Lipschitzian boundary chart, cf.
\cite[Ch.\ 1.1.9]{mazya}) or \cite[Def. 1.2.1.2]{grisvard85}.
\end{itemize}
Moreover, $ \Pi \subset \Omega$ is a Lipschitz surface (not necessarily connected) which forms a $(d-1)$-set, cf.
\cite[Ch. II/Ch. VIII.1]{jonsson}, and $\sigma$ is the surface measure on $\Gamma \cup \Pi$, cf. \cite[Ch.~3.3.4C]
{evans:gariepy92} or \cite[Ch.~3.1]{coerc} (being identical with the restriction of the 2-dimensional Hausdorff measure
to this set). 
\end{assu}
This defines the general geometric framework that is restricted implicitly
later on by Assumption \ref{assu:iso}. We are convinced that this
setting is sufficiently broad to cover (almost) all relevant semiconductor geometries -- in particular, referring to the arrangement of
$D$ and $\Gamma$. Please see also the more elaborate Remark \ref{r-kommentar} on this topic below.
%
\subsection[Notation]{Notation}
\label{sec:general}
%
For a Banach space $X$ we denote its norm by $\Vert{\cdot}\Vert_X$.
 $\mathbf X$ denotes the direct sum $X{\oplus}X$ of $X$ with itself. $\mathcal{L}(X;Y)$ is the space of
linear, bounded operators from the Banach space $X$ into the Banach space $Y$. We abbreviate
$\mathcal{L}(X):=\mathcal{L}(X;X)$. If $Z$ is a Banach space and $Z^*$ the space of (anti)linear forms
on $Z$, then  $\lrdual{\cdot}{\cdot}_Z$ always denotes the (anti)dual pairing between 
$Z$ and $Z^*$.\\
The (standard) notation $[X,Y]_\theta$, $(X,Y)_{\theta,r}$, respectively, is used for the complex, respectively real
interpolation spaces of $X$ and $Y$ with indices $\theta \in ]0,1[$, $r \in [1,\infty]$.
If $v$ is a function on an interval $J=]0,T[$ taking its values in a Banach space $X$, 
then $\dot{v}$ indicates its derivative in the sense of $X$-valued distributions,
cf. \cite[Ch. III.1.1]{amannbuch}. 
%
\subsection[Function spaces]{Function spaces}
\label{sec:spaces}
%
We exemplarily define spaces of functions on the bounded domain $\Omega\subset\mathbb{R}^3$
and on its boundary. 
In the following, we (mostly) 
write $L^2$ instead of $L^2(\Omega)$ and use this convention for all spaces of functions, functionals
 or distributional objects on the bulk domain $\Omega$. \\
If $p\in[1,\infty]$, then $L^p$ is the usual real Lebesgue space on $\Omega$. 
$H^{\theta,q}$ denotes the space of real Bessel potentials (cf. \cite[Ch. 4.2]{triebel}), which coincides with the usual Sobolev space $W^{1,q}$ on $\Omega$ in case of 
$\theta =1$, cf. \cite[Ch. 2.3.3]{triebel}.
$H^{\theta,q}_{\Diri}$ denotes the closure of
\begin{equation*}
C^\infty_D =  \left\{
    \psi|_{\Omega}
    \with
    \psi \in C^\infty_0(\mathbb{R}^3)
    , \;
    \supp \psi 
    \cap \Diri =\emptyset 
  \right \}
  ,
\end{equation*}
in $H^{\theta,q}$, which means that 
 $H^{\theta,q}_{\Diri}$ consists of all elements of
$W^{1,q}$ with vanishing trace on $\Diri$, -- if the trace exists, compare \cite[Thm. 3.7/Corollary 3.8]{jonsson}.
$H^{-\theta,q}_{\Diri}$ denotes the dual of $H^{\theta,q'}_{D}$, where $\frac{1}{q}+\frac{1}{q'}=1$. 
The requirements on $\Omega$ and on $D$ imply the usual interpolation properties within the $\{W^{1,q}_D\}_q$- and 
$\{W^{-1,q}_D\}_q$-scales, cf. \cite{jonsson}.\\
If $Z$ is a Banach space and $A$ is a linear and closed operator in $Z$, then we denote its domain of definition
by $dom_Z(A)$.
%

\subsection[Elliptic operators I: realistic assumptions on material geometry and coefficients]
{Weak elliptic operators in non-smooth settings}

Before defining the elliptic operators relevant for \eqref{vanRoos}, we introduce 
the following symmetry and ellipticity conditions:
\begin{defn} \label{d-coeff}
A bounded, measurable, elliptic coefficient function $\rho$
on $\Omega$ that takes its values in the set of symmetric $3\times 3$-matrices, is called an 
\emph{elliptic coefficient function}. Bounded and elliptic means the existence of two 
constants $\low{\rho}$ and $\upp{\rho}$ such that   
 \begin{equation*}
     \low{\rho} |\mathrm y|^2 \leq
    {(\rho(x)\mathrm y)}\cdot{\mathrm y}
    \leq \upp{\rho}|\mathrm y|^2
    ,\quad \text{for a.a. }x \in \Omega,\quad\text{for all }\mathrm y \in \mathbb{R}^3.
  \end{equation*}
\end{defn}

\begin{assu} \label{a-coercitiv}
			\renewcommand{\labelenumi}{\roman{enumi})}
			\begin{enumerate}
\item
The dielectric permittivity $\varepsilon$ and 
 the mobilities $\mu_k$, $k=1,2$ are elliptic coefficient functions.
\item
We assume that either the boundary measure of the Dirichlet boundary part
$D$ is positive or $\varepsilon_\Gamma$ is strictly positive on a subset
of $\Gamma$ which has positive boundary measure.
Physically spoken, the device has a Dirichlet contact or part of its surface has a positive
capacity.  
\end{enumerate}
\end{assu}

Considering the coefficient functions $\varepsilon$ and $\varepsilon_{\Gamma}$ from now on as fixed, we
define the Robin Poisson operator 
$\widehat{P}: W^{1,2} \rightarrow W^{-1,2}_{D}$ by 
\begin{equation} \label{e-poiss}
\dual{\widehat{P}\psi}{\vartheta}_{W^{1,2}_D}
=  \int_{{\Omega}}
    {(\varepsilon \nabla \psi)}\cdot{\nabla \vartheta}
    \der{x}
    +
    \int_{\Gamma}
    \varepsilon_{\Gamma}  \psi \;\vartheta
    \der{{\sigma}}
    ,
    \;   \psi \in {W}^{1,2},
    \;
    \vartheta \in {W}_{D}^{1,2}.
\end{equation}
Correspondingly, $P$ denotes the restriction of  $\widehat{P}$  to the domain
  ${W}_{D}^{1,2}$.\\ By a slight abuse of notation, $P$  may also denote the maximal restriction of $P$ to any 
range space which continuously embeds into ${W}_{D}^{-1,2}$.
\begin{rem} \label{r-coerc}
Assumption \ref{a-coercitiv} assures that the Poisson operator is coercive, cf. \cite{coerc} and \cite{egert}, and, hence, 
$P:W^{1,2}_D \to W^{-1,2}_D$ is a topological isomorphism.
\end{rem}
Let $\rho$ be an elliptic coefficient function on $\Omega$. Then we define the elliptic operator
$A_{\rho}: W_D^{1,2} \rightarrow W^{-1,2}_{D} $ by 
  \begin{equation} \label{e-ellippp}  
 \dual{ A_{\rho}\psi}{\vartheta}_{W^{-1,2}_D}
 =  \int_{{\Omega}}
      {(\rho \nabla \psi)}\cdot{\nabla \vartheta}
      \der{x}, \quad
      \;   \psi, \vartheta \in {W}_{D}^{1,2},
  \end{equation}
 which may also denote its maximal restriction to a smaller range space. The operator $\widehat{A}_\rho$ is defined accordingly, acting on $W^{1,2}$. Of particular interest is the 
case $\rho= \eta \mu_k$, with $\eta$ a bounded, strictly positive scalar function.

\par 
For our analysis of van Roosbroeck's system, the following assumption is crucial.
\begin{assu} \label{assu:iso}
There is a common integrability exponent $q \in ]3,4[$, such that the operators
\begin{equation} \label{e-iso1}
P:W^{1,q}_D \to W^{-1,q}_D
\end{equation}
and 
\begin{equation} \label{e-iso02}
A_{\mu_k}:W^{1,q}_D \to W^{-1,q}_D, \; k=1,2,
\end{equation} 
are topological isomorphisms.
\end{assu}
\begin{rem} \label{r-kommentar}
			\renewcommand{\labelenumi}{\roman{enumi})}
			\begin{enumerate}
\item
Gajewski and Gr\"oger have already observed in their pioneering
paper \cite{gajewski:groeger89} that a condition like this -- in 1989 being an ad hoc assumption -- 
would lead to a more satisfactory 
analysis of van Roosbroeck's system, compare also the discussion in \cite{veiga}.
\item
If \eqref{e-iso1} or \eqref{e-iso02} is a topological isomorphism for a $q >2$, then
this property remains true for all $\tilde q \in [2,q[$ by Lax-Milgram and 
interpolation, cf. \cite{jonsson}, so the set of such $q$s above $2$ always forms an interval. Thus, it is actually sufficient
 to assume that each of the operators in Assumption \ref{assu:iso} is an isomorphism for some $q>3$. 
Moreover, if $A_\rho:W^{1,q}_D \to W^{-1,q}_D$ is a topological isomorphism, then this property is 
maintained for coefficient functions $\eta \rho$, if the scalar function $\eta$ is strictly 
positive and uniformly continuous on $\Omega$, cf. \cite[Ch. 6]{disser}.
\item
	Assumption \ref{assu:iso} is fulfilled by very general classes of ``layered'' structures
and  additionally, if $D$ and its complement do not meet in a ``too wild'' manner, cf. \cite{RoChristJo} for the most
relevant model settings. A global framework has recently been established in \cite{disser}. 
However, Assumption \ref{assu:iso} indeed restricts the class of admissable coefficient
functions $\varepsilon$ and $\mu_k$. For instance, it is typically not satisfied if three or more different
 materials meet at one edge.
\item
Assumption \ref{assu:iso} also includes interesting geometric constellations that are not covered in \cite{disser}. A relevant example are \emph{buried contacts}, cf. Figure \ref{fig-buried}. The characteristic property of these constellations is that they
touch themselves `from the other side' -- but only at the Dirichlet boundary part $D$. In particular, they
need not be Lipschitz domains. 
	\begin{figure}[htbp]
	\centerline{\includegraphics[height=2in]{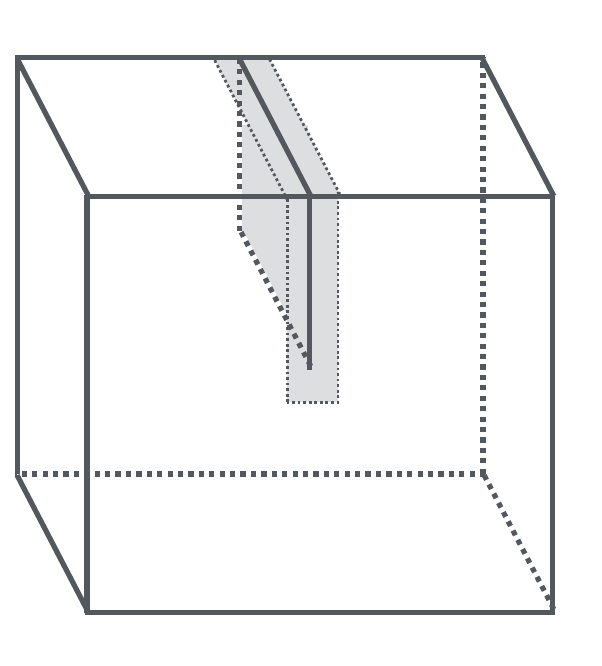}
	}
	\caption{\label{fig-buried}{Sketch of an idealized buried contact as an example of an admissible geometric setting. Dirichlet boundary conditions hold at the contact, i.e.\ on the shaded areas at the inner (buried) surface and close to its outer contact line. 
	}}
	\end{figure}

\item
Note that it is typically not restrictive to assume that all three differential operators
provide topological isomorphisms, if one of them does, since this
property mainly depends on the (possibly) discontinuous
coefficient functions versus the geometry of $D$. This is determined by the material properties of the device on
 $\Omega$, i.e., the coefficient functions $\mu_1, \mu_2, \varepsilon$ 
will often exhibit similar discontinuities and degeneracies.
\end{enumerate}
\end{rem}

%
   \subsection[Gradient-dependent and interfacial recombination]{Assumptions on recombination terms  in \eqref{CuCo-eq}}   
   \label{Reco}
%
For the recombination terms $\reaction^\Omega,$ $\reaction^\Pi, r^\Gamma$ in \eqref{CuCo-eq}, we require the following. 
\begin{assu} \label{assu:recomb}
Let $q$ be as in Assumption \ref{assu:iso}.
We assume that the reaction term in the  bulk, $\reaction^\Omega$,
is a locally Lipschitzian mapping
\begin{equation*}
\reaction^\Omega:\mathbf{W}^{1,q} \times W^{1,q} \times 
{\mathbf W}^{1,q} \ni (u, \varphi, \Phi) \mapsto r^\Omega(u, \varphi, \Phi) \in L^{\frac {q}{2}}.
\end{equation*}
\end{assu}

\begin{assu}\label{recombGamma}
We assume that the reaction term on $\Gamma$, $\reaction^\Gamma$, is a locally Lipschitzian mapping
\begin{equation*}
	r^\Gamma:
	\mathbf  W^{1,q} \times W^{1,q} \times \mathbf W^{1,q} \ni (u, \varphi, \Phi) \mapsto 
r^\Gamma(u, \varphi,\Phi) \in L^{4}(\Gamma, \sigma).
		\end{equation*}
 The same assumption holds, mutatis mutandis, for $r^\Pi$.
\end{assu}
In particular, the recombination terms introduced in 
\eqref{e-shockleyrh} and \eqref{e-auger} are included.
It is nontrivial to see that the Avalanche generation term, depending on the electric
field and the currents also satisfies Assumption \ref{assu:recomb}. Since the generality of Assumption
\ref{assu:recomb} causes considerable functional analytic effort in the analysis of the system, we give a 
detailed proof that Avalanche generation \eqref{e-0} is indeed included:
It is straightforward to check that the mappings
\begin{equation*} 
L^\infty \times W^{1,q} \times W^{1,q} \ni (u_k,\varphi,\Phi_k) \mapsto \bigl (\nabla \varphi, j_{k}(u_k,\Phi_k)\bigr ) \in \mathbf 
L^q(\Omega;\R^3)
\end{equation*}
are boundedly Lipschitzian.
If $\nabla \varphi$ and $j_k$ are orthogonal to each other, in order to give the expression in \eqref{e-0} a precise meaning, 
we introduce the function $\varkappa:\R^3 \times \R^3\to [0,\infty[$ with
\begin{equation} \label{e-hyper}
\varkappa (\mathfrak e,\mathfrak j) =\begin{cases} 0, \quad \text {if} \quad \mathfrak e \cdot \mathfrak j =0,\\
|\mathfrak j| \exp(\frac {-a}{|\mathfrak e \cdot
\frac {\mathfrak j}{|\mathfrak j|}|}),\quad \text{otherwise,}
\end{cases}
\end{equation}
for $a>0$.
It then suffices to show the following result. 
\begin{lem} \label{l-Lipsch}
The mapping 
\begin{equation*}
 L^q(\Omega; \R^3) \times L^q(\Omega;\R^3) \ni (\mathfrak e, \mathfrak j) \mapsto \varkappa \bigl (\mathfrak e (\cdot),\mathfrak j(\cdot) \bigr) 
\end{equation*}
takes its values in the space $L^{q}(\Omega)$ and
admits 
the Lipschitz estimate
\begin{equation} \label{e-locLipsch}
\Vert \varkappa (\mathfrak e_1 , \mathfrak j_1  ) - \varkappa (
\mathfrak e_2, \mathfrak j_2  ) \Vert_{q/2}
\le
 (|\Omega|^{1/q} + 2 L_a \| \mathfrak{e}_1 \|_{q}) \|\mathfrak{j}_1 -\mathfrak{j}_2 \|_{q} +
L_a \|\mathfrak{j}_2\|_{q} \|\mathfrak{e}_1 -\mathfrak{e}_2 \|_{q}, 
\end{equation}
in $L^{q/2}(\Omega;\R^3)$, where $\| \cdot \|_{q/2}, \| \cdot \|_q $ are the norms in 
$L^{q/2}(\Omega;\R^3)$, $ L^q(\Omega;\R^3)$, respectively, and where $L_a = \frac{4}{e^2a} < \frac{0.542}{a}$. 
\end{lem} 

\begin{proof}
	For $\mathfrak{e}, \mathfrak{j} \in \R^3$, we consider the function $\varkappa$ in \eqref{e-hyper} as composed of the functions $f_a \colon [0,+\infty[ \ni t \mapsto e^{\frac{-a}{t}} $ and $\varpi \colon \R^3 \times \R^3 \ni (\mathfrak{e}, \mathfrak{j}) \mapsto |\mathfrak{e}\cdot \frac{\mathfrak{j}}{\Vert \mathfrak{j} \Vert}|$.Regarding $f_a$, note that it is analytic on $]0,+\infty[$, bounded by $1$, and has Lipschitz constant $L_a = \frac{4}{e^2 a}$, and the last two properties extend into $0$. To show the Lipschitz estimate, consider $\mathfrak e_1, \mathfrak e_2, \mathfrak j_1, \mathfrak j_2 \in \mathbb{R}^3$. If $\mathfrak j_1 = \mathfrak j_2 = 0$, then the estimate is trivial. Without loss of generality, let $\mathfrak j_1 \neq 0$. Regarding $\varpi$, we estimate
	\begin{equation*}
		|\varpi(\mathfrak{e}_1, \mathfrak{j}_1) - \varpi(\mathfrak{e}_1, \mathfrak{j}_2)| \leq \Vert \mathfrak{e}_1 \Vert \frac{2}{\Vert \mathfrak{j}_1 \Vert} \Vert \mathfrak{j}_1 - \mathfrak{j}_2 \Vert, 
		\end{equation*} 
		and
		\begin{equation*}
			|\varpi(\mathfrak{e}_1, \mathfrak{j}_2) - \varpi(\mathfrak{e}_2, \mathfrak{j}_2)| \leq \Vert \mathfrak{e}_1 - \mathfrak{e}_2 \Vert.
			\end{equation*} 
			Thus, we obtain
			\begin{align*}
			|\varkappa(\mathfrak{e}_1, \mathfrak{j}_1) - \varkappa(\mathfrak{e}_2, \mathfrak{j}_2)| & \leq |\varkappa(\mathfrak{e}_1, \mathfrak{j}_1) - \varkappa(\mathfrak{e}_1, \mathfrak{j}_2)| + |\varkappa(\mathfrak{e}_1, \mathfrak{j}_2) - \varkappa(\mathfrak{e}_2, \mathfrak{j}_2)| \\
			& \leq \Vert \mathfrak{j}_1 \Vert |f_a(\varpi(\mathfrak{e}_1, \mathfrak{j}_1)) - f_a(\varpi(\mathfrak{e}_1, \mathfrak{j}_2)) | \\
			& \qquad + \Vert \mathfrak{j}_1 - \mathfrak{j}_2 \Vert | f_a (\varpi (\mathfrak{e}_1, \mathfrak{j}_2))| + L_a \Vert \mathfrak{e}_1	 - \mathfrak{e}_2 \Vert \\
		& \leq  (2L_a \Vert \mathfrak{e}_1 \Vert +1) \Vert \mathfrak{j}_1	 - \mathfrak{j}_2 \Vert + L_a \Vert \mathfrak{e}_1	 - \mathfrak{e}_2 \Vert. 			\end{align*}
				The estimate \eqref{e-locLipsch} now follows from H\"older's inequality. 
	\end{proof}


\subsection[Elliptic operators II: the domains of fractional powers]{Elliptic operators II: the domains of fractional powers}
\label{powers}
We choose an abstract formulation for the system that intricately solves the analytical problems arising from combining non-smoothness of material and geometry and nonlinearity of the dynamics. This gives rise to some technicalities in the proof. For example, on one hand, our techniques heavily rest on
\emph{complex} methods; this is in particular the instrument to provide exact descriptions for
the domains of fractional powers of the elliptic operators involved. On the other hand, the system is intrisically a real one -- of course, we are (only) interested in  real solutions. 
In this subsection, we consider complex Banach spaces and complexifications of the elliptic operators $A_\rho$.
In order to avoid further indices, the complex objects are denoted analogously to the real ones,
only furnished by an underline. 

Let $\rho$ be an elliptic coefficient function on $\Omega$. Then we define the elliptic operator
$\underline A_\rho: \underline W_D^{1,2} \rightarrow \underline W^{-1,2}_{D} $ by 
  \begin{equation*} 
 \dual{\underline A_\rho \psi_1}{\psi_2}_{\underline W^{-1,2}_D}
  = \int_{{\Omega}}
      {(\rho \nabla \psi_1)}\cdot{\nabla \overline \psi_2}
      \der{x}, \quad
      \;   \psi_{1}, \psi_{2} \in \underline {W}_D^{1,2},
  \end{equation*}
We show that the isomorphism property \eqref{e-iso02}
 transfers to the complex spaces.
\begin{lem} \label{l-isocomplex}
If $\rho$ is a real, elliptic coefficient function, such that 
\begin{equation} \label{e-isoreell}
A_\rho:W^{1,q}_D \to W^{-1,q}_D
\end{equation}
is a topological isomorphism,  then 
\begin{equation} \label{e-isoreco}
\underline A_\rho:\underline  W^{1,q}_D \to \underline  W^{-1,q}_D
\end{equation}
is a topological isomorphism.
\end{lem}
\begin{proof}
We define  a *-operation in $\underline  W^{-1,q}_D$ by setting
 $\langle f^*|\psi \rangle_{\underline  W^{-1,q}_D} :=\overline 
{\langle f| \overline \psi \rangle_{\underline  W^{-1,q}_D}},$ , for $\psi \in \underline W^{1,q'}_D$.  Evidently, one has $f=\frac {f +f^*}{2} + i \frac {f -f^*}{2i}$ and both $f_1:=\frac {f +f^*}{2}$ and $f_2:=\frac {f -f^*}{2i}$ attain real values for real functions
 $\psi \in W^{1,q'}_D$ . Hence, $f_1, f_2$ may be viewed as elements of
the real space $W^{-1,q}_D$. 
Moreover, since $\underline A_\rho^{-1}$ transforms real elements $f=f^* \in W^{-1,q}_D$ into real functions, the 
isomorphism property \eqref{e-isoreell} carries over to the one in \eqref{e-isoreco}.
\end{proof}

In case of smooth data (smooth domains, coefficients and absence of mixed boundary conditions) the determination of the
domains of fractional powers is classical, cf. \cite{seeley}. In our situation, this does not work, but the subsequent
powerful results from \cite{ausch} apply.
\begin{prop}\label{p-ausch}
Assume $q \ge 2$ and let $\rho$ be an elliptic coefficient function on $\Omega$. Then
	\renewcommand{\labelenumi}{\roman{enumi})}
			\begin{enumerate}
\item 
$(\underline A_\rho+1)^{\frac {1}{2}}$ provides a topological isomorphism of $\underline  L^q$
 and $\underline  W^{-1,q}_D$,
\item
the operator $\underline A_\rho+1$ is positive on both spaces, $\underline  L^q$ and $\underline  W^{-1,q}_D$, i.e.\ it 
satisfies resolvent estimates
of the kind 
\begin{equation} \label{e-reslv}
\|(\underline A_\rho +1+\lambda)^{-1}\|_{\mathcal L(\underline  L^q)} \le \frac {1}{1+\lambda}, \quad \|(\underline A_\rho +1+\lambda)^{-1}
\|_{\mathcal L(\underline  W^{-1,q}_D)}
 \le \frac {c}{1+\lambda}, 
\end{equation}
for all $ \lambda \in [0,\infty[$ and some constant $c$, cf \cite[Ch. 1.14]{triebel}.
In consequence, all fractional powers are well-defined, cf. \cite[Ch. 1.15]{triebel}.			%
\item
the operator $\underline A_\rho+1$ admits \emph{bounded purely imaginary powers} on $\underline  W^{-1,q}_D$, i.e.\ one has
\[
\sup_{\tau \in [0,1]} \|(\underline A_\rho+1) ^{i \tau} \|_{\mathcal L(\underline  W^{-1,q}_D)} < \infty.
\]
\end{enumerate}
\end{prop}
\begin{proof}
i) is the main result of \cite{ausch}, see Thm. 5.1. Regarding ii), it is well-known that, under the above conditions,
 $\underline A_\rho$ generates a strongly continuous contraction semigroup on every $\underline  L^p$,
 $ p \in ]1,\infty[$, cf. \cite[Thm. 4.28]{Ouh05}. Thus, the first resolvent estimate in \eqref{e-reslv} 
follows by the Hille-Yosida theorem, cf \cite[Thm. X.47a]{reed2}.
The second estimate is deduced from the first by i). Finally, iii) is proved in \cite[Ch. 11]{ausch}.
\end{proof}

%
\subsection[Inhomogeneous data]{Inhomogeneous data}
\label{sec:diribv}
%
For setting up the Poisson and current--continuity equations in
appropriate function spaces, we split the unknowns into two parts,
where one part each represents the inhomogeneous Dirichlet boundary values
${\varphi}_D$ and $\Phi_k^D$, $k=1,2$. 

\begin{assu} \label{assu:poi-diri-bv}
	The data $d, \varphi_\Gamma, \varphi_D$ and $\Phi_k^D$ in \eqref{vanRoos} are such that
\renewcommand{\labelenumi}{\roman{enumi})}
  \begin{enumerate}
\item 
the doping $d$ is either contained in $ W^{1,\infty}(J;L^{q/2})$ or it is independent of time and satisfies $d \in W^{-1,q}_D$, 
which would include dopings concentrated on surfaces, cf.\ \cite{naz} \cite{drumm}.

\item 
the Robin boundary value $\varphi_\Gamma$ satisfies $\varphi_\Gamma \in W^{1,\infty}(J;L^4(\Gamma,\sigma))$, 

\item
there are functions 
$\varphi^D, \Phi^d_k \in W^{1,\infty}(J;L^{q/2})\cap L^\infty(J;W^{1,q})$ that also satisfy $\widehat{A}_{\mu_k} \Phi^d_k,
 \widehat{P} \varphi^D \in L^\infty(J;L^{q/2})$  such that $\varphi^D(t)\vert_D = \varphi_D (t)$ and $ \Phi^d_k(t)\vert_D
 = \Phi^D_k(t)$ in the sense of traces.
\end{enumerate}
\end{assu}

\begin{rem} \label{r-L4}
Note that we do not suppose that the function $\varphi_\Gamma$ takes its values in $L^\infty(\Gamma, \sigma)$
with regularity assumptions for the dependence on time. If there were a 
continuity requirement on the mapping $\overline{J} \ni t \mapsto 
\varphi_\Gamma(t,\cdot) \in L^\infty(\Gamma, \sigma)$, this would exclude an indicator function of a subset of
$\Gamma$ that moves  in $\Gamma$ over time. 
\end{rem}

\begin{rem} \label{data}
	The regularity Assumption \ref{assu:poi-diri-bv}\,iii) is easily satisfied for smooth $\varphi_D$ and $\Phi^D$. In view of the fact that $D$ is a $(d-1)$-set of Jonsson/Wallin (cf. 
	\cite[Ch. II]{jons}), we refer to
	\cite[Ch.~V]{jonsson} for examples of suitable extension and trace operators.
Note that additional time regularity of the data transfers to additional regularity of the solution, cf.\ Theorem \ref{mr} and Remark \ref{r-Pi} below. 
	\end{rem}

  With Assumption \ref{assu:poi-diri-bv}, define 
  \begin{equation*}
	  \varphi_d(t) = P^{-1}(d+\varphi_\Gamma)(t) + (\mathrm{Id} - P^{-1}\widehat{P}) \varphi^D)(t) \in W^{-1,q}. 
	  \end{equation*} 
	  Then $\varphi_d$  solves 

	  \begin{equation*} 
	  	\left\{ \begin{array}{rcll}
	  	-\mathrm{div}\,\varepsilon \nabla \varphi_d & = & d, & \text{in } \Omega, \\
	  	\varphi_d & = & \varphi_D, & \text{on } D, \\
	  	\nu \cdot (\varepsilon \nabla \varphi_d) + \varepsilon_\Gamma \varphi_d & = & \varphi_\Gamma, & \text{on } \Gamma, 
	  	\end{array} \right.
	  	\end{equation*} 
 and the split 
\begin{equation}\label{splitPot}
	\varphi = \varphi_d + \tilde{\varphi}
	\end{equation}
	gives a solution $\varphi$ of \eqref{Poisson-eq} with 
	\begin{equation}\label{tildebyu}
		\tilde{\varphi} = P^{-1}(u_1 - u_2).
		\end{equation} 
	 
	 For the quasi Fermi levels $\Phi_k$, in the following, we use the direct split 
	 \begin{equation}\label{splitPhi}
\Phi = \phi + \Phi^d,	 
		 \end{equation}
		 so that, in particular, $\phi(t) \in \mathbf{W}^{1,q}_D$ is equivalent to $\Phi(t) \in \mathbf{W}^{1,q}$ and 
$\Phi(t)\vert_D = \Phi^D(t)$.


\section[Abstract formulation of the system]
{Abstract formulation of \eqref{vanRoos}}

\label{sec:nl-reform}
In this section, we rewrite the van Roosbroeck system as a quasilinear abstract Cauchy problem for the homogeneous 
 quasi Fermi levels $\phi_1, \phi_2$, 
 \begin{equation} \label{ACP}
 		\dot{\phi}(t) + \mathcal{A}(t,\phi(t))\phi(t)  = \mathcal{R}(t,\phi(t)) \in \mathbf X, 
 		\end{equation}
with initial condition $\phi(0) = \Phi_0 - \Phi^d(0)$.
In the next subsection, we motivate and define the Banach space $\mathbf X$ 
 -- being a rather `unorthodox' one --
in which the problem is set. It becomes clear why the requirements due to the combination of non-smoothness and non-linearity
 of the system 
 do not allow us to use an $L^p$- or an $W^{-1,2}_D$-space. We then prove the preliminary properties of the space $\mathbf{X}$
 that justify its choice and are needed in the following. \\
 To derive \eqref{ACP}, we eliminate the electrostatic potential $\varphi$ from
the continuity equations. Replacing the carrier densities $u_1$ and
$u_2$ on the right hand side of Poisson's equation by \eqref{eq:densities}/\eqref{e-relation} -- thereby 
taking into account \eqref{splitPot} and \eqref{splitPhi} -- one obtains
 a nonlinear Poisson equation for $\tilde{\varphi}$. In Subsection \ref{sec:nl-poisson}, we solve this equation in its dependence of 
 prescribed quasi Fermi levels $\Phi \in \mathbf{W}^{1,q}$. This way of nesting the equations is also used in numerical
schemes for the van Roosbroeck system. It is due to Gummel \cite{gummel64} and was the first reliable numerical
technique to solve these equations for carriers in an operating semiconductor device structure.\\
Finally, in Subsection \ref{sec:der}, we derive the abstract formulation of type \eqref{ACP}. 

\subsection{Choice of the ambient space $\mathbf{X}$}

We discuss structural and regularity properties of the unknowns $u,\varphi, \Phi$ of the transient semiconductor equations in
 \eqref{vanRoos} to motivate the choice of $\mathbf{X}$.

\begin{itemize}
\item 
In view of the 
jump condition on the surface $\Pi$ on the fluxes $j_k$ 
in \eqref{CuCo-eq}, it cannot be expected that $\mathrm{div}\,j_k$ is a function.
This excludes spaces of type $L^p$, cf.\ Remark \ref{r-Pi}. 
In addition, with the choice of a space $\mathbf{X}$
that includes distributional objects, 
the inhomogeneous Neumann conditions $r^\Gamma$ in the
 current-continuity equations \eqref{CuCo-eq} and the surface recombination term $r^\Pi$ can be included in the right-hand side
 of \eqref{ACP}, cf.\ Lemma \ref{t-propX}.
\item 
For our analysis, we require an adequate parabolic theory for the divergence operators on $\mathbf X$. Due to the \emph{non-smooth geometry}, the \emph{mixed boundary conditions} and \emph{discontinous 
coefficient functions}, this is nontrivial. The first crucial 
point is that the operators
have to satisfy maximal parabolic regularity on $\mathbf X$, with a domain of definition that does not change,
 cf.\ Lemma \ref{l-maxparpert}.
\item For the handling of `squares' or other functions of gradients in the Avalanche and other recombination terms,  
 the Banach space $X$ should be sufficienlty `small' so that the parabolic time-trace space, cf. Theorem \ref{mr}, embeds
 into $W^{1,q}$, cf.\ Corollary \ref{ttrace}. This excludes spaces of type $W_D^{-1,r}$. With this strategy, at the same time,
 the space needs to be sufficiently large for the embedding $L^{q/2} \hookrightarrow X$ to hold, cf. Lemma \ref{t-propX}.
\item Finally, the dependence $\eta \mapsto A_{\eta \rho}$, cf.\ \eqref{e-ellippp}, should be well-behaved in the sense that
 it should be Lipschitz with respect to functions  
$\eta$ in the parabolic time-trace space, cf.\ Lemma \ref{t-multipl}. 
\end{itemize}
With this discussion in mind, for $q>3$
the number from  Assumption \ref{assu:iso}, we define
\begin{equation*}
X:=[L^{q}, W^{-1,q}_D]_{\frac {3}{q}} \quad\text{ and }\quad \mathbf X:=X \oplus X.
\end{equation*}
Moreover, we put $\mathcal D_{\mu}:= dom_X(A_{\mu_1}) \oplus dom_X(A_{\mu_2})$, equipped with the graph norm.

\begin{rem} \label{r-complexinterp}
The complex interpolation functor applies to real spaces in the usual sense, following \cite[Ch. 2.4.2]{amannbuch}:  the spaces are complexified, then interpolated and then the `real part' is considered.
\end{rem}

We show that $X$
 and
$\mathbf X$, respectively, together with the occurring operators, 
possess the properties claimed in the discussion.

\begin{lem} \label{l-intreg}
Recall  that $\underline X=[\underline L^q,\underline W^{-1,q}_D]_\frac {3}{q}$. Assume that $\rho$ is an elliptic
 coefficient function, such that 
\begin{equation*} 
A_\rho:W^{1,q}_D \to W^{-1,q}_D
\end{equation*}
is a topological isomorphism. Then 
\renewcommand{\labelenumi}{\roman{enumi})}
\begin{enumerate}
\item
we have
$dom_{\underline  W^{-1,q}_D}\bigl ((\underline A_{ \rho}+1)^{\frac {1}{2}(1- \frac {3}{q})})\bigr )=  [\underline L^q, \underline W^{-1,q}_D]_{\frac {3}{q}}$, and
\item
the embedding
\begin{equation*}
\bigl (X,dom_{X}(A_{\rho})\bigr )_{\varsigma,\infty} \hookrightarrow W^{1,q}, \quad \text {if} \quad 
\varsigma \in ]\frac {1}{2}+\frac {3}{2q},1[.
\end{equation*}
\end{enumerate}
\end{lem}
\begin{proof}
i) According to Proposition \ref{p-ausch}, $\underline A_{\rho}+1$ is a positive operator on $\underline  W^{-1,q}_D$, 
possessing bounded purely imaginary powers. This gives, according to  \cite[Ch. 1.15.3]{triebel},
\begin{equation*}
dom_{\underline  W^{-1,q}_D}\bigl ((\underline A_{ \rho}+1)^{\frac {1}{2}(1-\frac {3}{q}}))\bigr )
=[\underline  W^{-1,q}_D,dom_{\underline  W^{-1,q}_D}\bigl((\underline A_{ \rho}+1)^\frac {1}{2}\bigr )
]_{1-\frac {3}{q}} 
\end{equation*}
\begin{equation*}
=[\underline  W^{-1,q}_D,\underline  L^q]_{1-\frac {3}{q}} =[\underline  L^q,\underline  W^{-1,q}_D]_{\frac {3}{q}}= \underline X.
\end{equation*}
ii) From i), it immediately follows that $\bigl (\underline A_{\rho} +1)^{\frac {1}{2}(1 + \frac {3}{q})}:
\underline W^{1,q}_D \to [\underline L^q,\underline W^{-1,q}_D]_\frac {3}{q}$ is a topological isomorphism; in other words
$dom_{\underline X}\bigl ((\underline A_{ \rho}+1)^{\frac {1}{2}(1 + \frac {3}{q})}\bigr ) =\underline W_D^{1,q}$.
Since $\underline A_\rho$ is -- by interpolation -- also a positive operator on $\underline X$, for 
$\varsigma \in ]\frac {1}{2}+\frac {3}{2q},1[$, we have
\[
(\underline X,dom_X(\underline A_\rho))_{\varsigma,\infty} \hookrightarrow (\underline X,dom_{\underline X}
(\underline A_\rho))_{\frac {1}{2}(1+\frac {3}{q}),1}
 \hookrightarrow dom_{\underline X}((\underline A_\rho+1)^{\frac {1}{2}(1+\frac {3}{q}) }) = \underline W^{1,q}_D,
\]
cf. 
\cite[Thm. 1.3.3 e)]{triebel} and \cite[Thm. 1.15.2]{triebel}). This proves the assertion in the complex case.
But $X$ is a real subspace of $\underline  X$ and $dom_{X}(A_{ \rho})$ is a real subspace of
 $dom_{\underline  X}(\underline A_{ \rho})$. So the real interpolation space 
$\bigl (X,dom_{X}(A_{ \rho})\bigr )_{\varsigma,\infty} $ must be embedded in the `real part' $W^{1,q}_D$ of 
$\underline  W^{1,q}_D$.\\
\end{proof}
\begin{cor} \label{c-interint}
Under Assumption \ref{assu:iso}, we obtain 
\begin{equation*}
\bigl (\mathbf X,\mathcal D_\mu\bigr )_{\varsigma,\infty} \hookrightarrow \mathbf W_D^{1,q}, \quad \text {if} \quad 
\varsigma \in ]\frac {1}{2}+\frac {3}{2q},1[.
\end{equation*}
\end{cor}
For convenience, we defined the recombination terms $r^\Gamma$ and $r^\Pi$ as $L^4(\Gamma,\sigma)$-valued 
and $L^4(\Pi,\sigma)$-valued, respectively,
since one has an intuitive understanding of this condition. Since the whole system will be considered in
the space $\mathbf X$, in the next result, we connect Assumption \ref{recombGamma} with spaces of type $\mathbf X$. 

\begin{lem} \label{t-propX}
	Let $\Gamma$ and $\Pi$ be as in Assumption \ref{spatial}. Then 
			\renewcommand{\labelenumi}{\roman{enumi})}
			\begin{enumerate}
\item
we have the embedding
$L^{\frac {q}{2} } \hookrightarrow X$, and
\item
there are continuous embeddings
\[
T^*_\Gamma \colon L^4(\Gamma,\sigma) \to X, \quad \text{and } \quad T^*_\Pi \colon L^4(\Pi,\sigma) \to X.
\]
given by the adjoints of the trace operators $T_\Gamma, T_\Pi$. 
\end{enumerate}
\end{lem}
\begin{proof}
i) According to the duality formula for interpolation \cite[Ch. 1.11.13]{triebel}, 
\[
[L^q,W^{-1,q}_D]_\theta =[L^{q'},W^{1,q'}_D]^*_\theta,
\]
and taking into account Remark \ref{r-complexinterp},
the assertion is equivalent to $[\underline L^{q'},\underline W^{1,q'}_D]_{\frac {3}{q}} \hookrightarrow \underline L^{(\frac {q}{2})' }$
. Exploiting the fact that
the spaces $\underline L^{q'}$ and $\underline W^{1,q'}_D$ admit a common extension operator to 
$\underline L^{q'}(\R^3)$ and $\underline W^{1,q'}(\R^3)$, respectively,
and the interpolation equality
\[
[\underline L^{q'}(\R^3),\underline W^{1,q'}_D(\R^3)]_{\frac {3}{q} } = \underline H^{\frac {3}{q},q'}(\R^3),
\]
one obtains,
in combination with the embedding $\underline H^{\frac {3}{q},q'}(\R^3) \hookrightarrow \underline L^{(\frac {q}{2})'}(\R^3)$, the 
first assertion.\\
ii) We prove the dual statements, i.e. the existence of trace mappings
\begin{equation} \label{e-trace}
T_\Gamma \colon X^*=[L^{q'}, W^{1,q'}_D]_{\frac {3}{q}}
\to L^{\frac {4}{3}}(\Gamma,\sigma), \quad \text{and} \quad
T_\Pi \colon X^*
\to L^{\frac {4}{3}}(\Pi,\sigma),
\end{equation} 
thereby again taking into account Remark \ref{r-complexinterp}.
In view of $q <4$, we have the inequalities $\frac {3}{q} > \frac {1}{q'} = 1-\frac {1}{q}$
and $q' > \frac {4}{3}$. We establish the first trace mapping in \eqref{e-trace}. First, one may localize the setting.
Then, thanks to the Lipschitz property of $\partial \Omega$ in a neighbourhood of $\overline \Gamma$,  the bi-Lipschitzian boundary charts can be applied, observing that the quality of $[\underline L^{q'}, \underline W^{1,q'}_D]_{\frac {3}{q}} \hookrightarrow
[\underline L^{q'},\underline  W^{1,q'}]_{\frac {3}{q}}$  is preserved under the corresponding transformation, so that the boundary part under consideration is `flat'. Hence, $[\underline L^{q'},\underline  W^{1,q'}]_{\frac {3}{q}} = \underline H^{\frac {3}{q},q'}$ can be applied locally, as in the half space case, 
\cite[Thm. 2.10.1]{triebel}, in order to see that the trace belongs to $\underline L^{q'}(\Gamma,\sigma)\hookrightarrow
 \underline L^{\frac {4}{3}}(\Gamma,\sigma)$.\\
We now establish the second trace mapping in \eqref{e-trace}. The starting point is the observation that
the properties of $\Omega$, cf.\ Assumption \ref{spatial}, allow for a continuous extension operator 
$\mathfrak E:\underline L^{q'}(\Omega) \to 
\underline L^{q'}(\R^3)$ the restriction of which to $\underline W_D^{1,q'}(\Omega)$ provides a continuous operator into
$\underline W^{1,q'}(\R^3)$, cf.\ \cite[Lemma 3.2]{ausch}. By interpolation, this gives a continuous extension operator 
\[
\hat {\mathfrak E}:\underline X^*=[\underline L^{q'},\underline  W^{1,q'}_D]_{\frac {3}{q}} \to 
[\underline L^{q'}(\R^3),\underline  W^{1,q'}(\R^3)]_{\frac {3}{q}}=\underline H^{\frac {3}{q}, q'}(\R^3).
\] Taking $\tau \in ]\frac {1}{q'},\frac {3}{q}[\neq
\emptyset$, we have the embedding $\underline H^{\frac {3}{q},q'}(\R^3) \hookrightarrow \underline W^{\tau,q'}(\R^3)$ into the 
corresponding Sobolev-Slobodetskii
space, cf. \cite[Ch. 4.6.1]{triebel}. Now we consider $\overline \Pi$, the closure of $\Pi$, instead of $\Pi$, 
and exploit that
$\overline \Pi$  is also a $2$-set, and $\overline \Pi \setminus \Pi$ is negligible with respect to the two-dimensional
Hausdorff measure, cf. \cite[Ch. VIII.1.1]{jons}). Then we use the trace mapping 
$\underline W^{\tau,q'}(\R^3) \hookrightarrow \underline L^{q'}(\overline \Pi,\sigma)\hookrightarrow \underline L^{\frac {4}{3}}
(\overline \Pi,\sigma)$, cf. \cite[Ch. V.1.1]{jons}. Finally, the definition of the trace (cf. \cite[Ch. I.2]{jons})
 as the limit of averages (pointwise a.e.\ with respect to $\sigma$) tells us that the trace of any function $\psi \in 
H^{\frac {3}{q},q'}$
on points of $\Pi$ is independent of the extension $\hat {\mathfrak E} \psi$, because $\Omega$ is open and 
$\Pi \subset \Omega$.
\end{proof}

\subsection{The nonlinear Poisson equation}

\label{sec:nl-poisson}
The aim of this subsection is to express the dependence of the homogeneous part of the electrostatic potential 
$\tilde{\varphi}$, cf. \eqref{splitPot}, in its dependence of the homogeneous quasi Fermi levels $\phi$. With 
$u_k(t) = \mathcal{F}_k\bigl (\phi(t) + \Phi^d_k(t) + (-1)^k(\tilde{\varphi}(t) + \varphi_d(t))\bigr )$ for some 
$\varphi_d(t), \Phi^d_k(t) \in L^\infty$ depending on the data, cf.\ Subsection \ref{sec:diribv}, this means 
that we need to solve the nonlinear Poisson problem 
\begin{equation}\label{nlPoiss}
	     P \tilde{\varphi}
	     =   \mathcal{F}_1(\omega_1 - \tilde{\varphi} )
	     -     \mathcal{F}_2(\omega_2 +\tilde{\varphi})
	\end{equation}
and to quantify the dependence of the solution of given functions $\omega \in \mathbf{L}^{\infty}$. 

With this analysis, we can then consider van Roosbroeck's equations as a quasilinear nonlocal problem in the unknowns $\Phi$ only. 
\begin{thm} \label{t-potentialto}
	\renewcommand{\labelenumi}{\roman{enumi})}
For every pair $\omega = (\omega_1,\omega_2) \in \mathbf L^\infty$ there is exactly one element $\tilde \varphi \in W^{1,q}_D$ that satisfies \eqref{nlPoiss}. We write $\tilde \varphi = \mathcal S(\omega)$. Then,
 			\begin{enumerate}
 \item
the mapping $\mathcal S:\mathbf L^\infty \to W^{1,q}_D$ is
continuously differentiable,
\item
the mapping $\mathcal S$, viewed between $\mathbf L^\infty$ and $L^\infty$, is globally Lipschitzian
with Lipschitz constant not larger than $1$, and
\item
$\mathcal S:\mathbf L^\infty \to W^{1,q}_D$ is boundedly Lipschitzian.
\end{enumerate}
\end{thm}
\begin{proof}
We first apply the implicit function theorem. In particular, define    
\begin{equation*} 
    \mathcal{K} \colon \mathbf{L}^{\infty} \times W^{1,q}_D \to W^{-1,q}_D 
	\end{equation*}
	by 
	\begin{equation*}
	\mathcal{K} (\omega, \tilde{\varphi}) 
    = 
    P \tilde{\varphi}
    -    \mathcal{F}_1(\omega_1 - \tilde{\varphi} )
    +     \mathcal{F}_2(\omega_2 +\tilde{\varphi}).
  \end{equation*}
We show that $\mathcal{K}$ is continuously differentiable and that the partial derivatives
with respect to $\tilde \varphi$ are  topological isomorphisms between ${W}_{D}^{1,q}$ and ${W}_{D}^{-1,q}$.
Then the level set $\mathcal{K}(\omega,\mathcal{S}(\omega)) 
    =
    0 $ implicitly defines the solution operator 
	\begin{equation}\label{defS}
	\mathcal{S} \colon \mathbf{L}^{\infty} \to W^{1,q}_D 
	\end{equation}
	of \eqref{nlPoiss} and $\mathcal{S}$ is continuously differentiable. 
  The partial derivatives of $\mathcal{K}$ are given by
  \begin{eqnarray*}
    \partial_{\tilde{\varphi}}
    \mathcal{K}(\omega,\tilde{\varphi})
    & = &
    P
    +
    \sum_{k=1}^2
     \mathcal{F}'_k
    (\omega_k 
    + (-1)^k
     \tilde{\varphi})
    \in
    \mathcal{L}(
    {W}_{D}^{1,q};
    {W}_{D}^{-1,q}
    ),
    \\
    \partial_{\omega_k}
    \mathcal{K}(\omega,\tilde{\varphi})
    & = &
 (-1)^k   \mathcal{F}'_k
    (\omega_k 
    + (-1)^k
    \tilde{\varphi})
    \in
    \mathcal{L}(
    L^{\infty};
    {W}_{D}^{-1,q}
    ),
  \end{eqnarray*}
  and they depend continuously on $\omega$ and $\tilde{\varphi}$. Note that here the expressions $\mathcal{F}'_k(\omega_k 
    + (-1)^k
    \tilde{\varphi}) \in L^{\infty}$ are to be understood as multiplication operators.\\
Consider the equation
\begin{equation} \label{e-laxmil}
 P \psi
    +
    \sum_{k=1}^2
     \mathcal{F}'_k(\omega_k + (-1)^k \tilde{\varphi})\psi=f \in {W}_{D}^{-1,q}.
\end{equation}                               %
Since 
\begin{equation*}
\sum_{k=1}^2 \mathcal{F}'_k(\omega_k + (-1)^k \tilde{\varphi})
\end{equation*} is a non-negative function in $L^\infty$, 
\eqref{e-laxmil} has a unique solution $ \psi \in {W}_{D}^{1,2}$
by the Lax-Milgram-Lemma. Moreover, $\sum_{k=1}^2  \mathcal{F}'_k
    (\omega_k     + (-1)^k \tilde{\varphi}) \psi$ is then contained in $L^2 \hookrightarrow 
{W}_{D}^{-1,q}$ and $P\colon{W}_{D}^{1,q}
\rightarrow {W}_{D}^{-1,q}$ is a topological isomorphism, so a 
rearrangement of terms in \eqref{e-laxmil} gives
$ \psi \in {W}_{D}^{1,q}$. It follows that $\partial_{\tilde{\varphi}} \mathcal{K}(\omega,\tilde{\varphi})$ is an 
isomorphism of ${W}_{D}^{1,q}$ and ${W}_{D}^{-1,q}$. This proves i). \\
ii) Given $\omega, \kappa \in \mathbf{L}^\infty$, consider the solutions $\tilde \varphi = \mathcal{S}(\omega) \in W^{1,q}_D$,
 $\tilde \psi = \mathcal{S}(\kappa) \in W^{1,q}_D$  -- each being even uniformly continuous. They satisfy 
	\begin{equation}\label{PoissLip}
		P(\tilde \varphi - \tilde \psi) = \mathcal{F}_1 (\omega_1 - \tilde \varphi) - \mathcal{F}_2(\omega_2 +
 \tilde \varphi) - \mathcal{F}_1(\kappa_1 - \tilde \psi) + \mathcal{F}_2(\kappa_2 + \tilde \psi) 
			\end{equation}
			in $W^{-1,q}_D \hookrightarrow W^{-1,2}_D$. Define 
			\begin{equation*}
				d = \max (\max (\Vert (\omega_1 - \kappa_1)^+\Vert_\infty, \Vert (\kappa_2 - \omega_2)^+\Vert_\infty),
 \max (\Vert (\kappa_1 - \omega_1)^+\Vert_\infty, \Vert (\omega_2 - \kappa_2)^+\Vert_\infty)),
				\end{equation*}
				and note that $d \leq \Vert \omega - \kappa \Vert_{\mathbf L^\infty}$. 
				Now let 
				\begin{equation*}
					h = \left\{ \begin{array}{ll}
					     \tilde \varphi - \tilde \psi - d, & \text{if }\tilde \varphi-\tilde \psi > d, \\
						 \tilde \varphi -\tilde \psi + d, & \text{if } \tilde \varphi - \tilde \psi < -d, \\
						 0, & \text{otherwise}.
					\end{array} \right.
					\end{equation*}
Taking into account the uniform continuity of $\tilde \varphi, \tilde \psi$, it is not hard to see that $h$ is an admissable test function
 in $W^{1,2}_D \cap L^\infty$.
Denote by $\Omega_+ = \{ x \in \Omega : h(x) > 0\}$, $\Omega_- = \{ x \in \Omega : h(x) < 0\}$ the (open) subsets of $\Omega$
 where $h$ is positive or negative, respectively. We apply \eqref{PoissLip}
 to $h$, cf. \eqref{e-poiss}, 
					\begin{equation*}
						\begin{array}{rl}
&\int_\Omega (\varepsilon \nabla h)\cdot\nabla h \,\mathrm{d}x 
+\int_\Gamma \varepsilon_\Gamma (\tilde \varphi - \tilde \psi) h \,\mathrm{d}\sigma \\
= &  \int_{\Omega_+} (\mathcal{F}_1 (\omega_1 - \tilde \varphi) - \mathcal{F}_1 (\kappa_1 - \tilde \psi)) h \,\mathrm{d}x - 
\int_{\Omega_+} (\mathcal{F}_2 (\omega_2 + \tilde \varphi) - \mathcal{F}_2 (\kappa_2 + \tilde \psi) h \,\mathrm{d}x \\
& + \int_{\Omega_-} (\mathcal{F}_1 (\omega_1 - \tilde \varphi) - \mathcal{F}_1 (\kappa_1 - \tilde \psi)) h \,\mathrm{d}x -
 \int_{\Omega_-} (\mathcal{F}_2 (\omega_2 + \tilde \varphi) - \mathcal{F}_2 (\kappa_2 + \tilde \psi) h \,\mathrm{d}x.
\end{array}
\end{equation*}
Clearly, the first addend on the left-hand-side is non-negative. Secondly, the function $(\tilde \varphi - \tilde \psi) h$
is non-negative on $\Omega$, so its trace on $\Gamma$ is also non-negative  a.e. with respect to $\sigma$. On the other hand, by the definition of $d$ and $h$
and the monotonicity of $\mathcal{F}_k$, all four terms on the right-hand-side are non-positive. It follows that
$h\equiv 0$ and thus 
\begin{equation*}
\Vert \tilde \varphi - \tilde \psi \Vert_{L^\infty} \leq d \leq \Vert \omega -  \kappa \Vert_{\mathbf L^\infty}, 
\end{equation*} 
which proves ii).\\
iii) is a direct consequence of re-investing ii) into \eqref{PoissLip}, where 
						\begin{align*}
& \Vert \tilde \varphi - \tilde \psi \Vert_{W^{1,q}_D}\\
& \leq \Vert P^{-1} \Vert_{\mathcal{L}(L^\infty;W^{1,q}_D)} \Vert \mathcal{F}_1 (\omega_1 - \tilde \varphi) - 
\mathcal{F}_2(\omega_2 + \tilde \varphi) - \mathcal{F}_1(\kappa_1 - \tilde \psi) + \mathcal{F}_2(\kappa_2 + \tilde \psi) \Vert_{\infty} \\
& \leq C_M (\Vert \omega - \kappa \Vert_{\mathbf L^\infty} + \Vert \tilde \varphi - \tilde  \psi \Vert_\infty)\\
& \leq 2 C_M \Vert \omega - \kappa \Vert_{\mathbf L^\infty},
							\end{align*}
							where the constant $C_M > 0$ depends on the local 
Lipschitz constants of $\mathcal{F}_k$ with respect to bounded sets of parameters $\Vert \omega \Vert_{\mathbf L^\infty},
 \Vert \kappa \Vert_{\mathbf L^\infty} < M$. 
	\end{proof}

\begin{rem}\label{r-groeger}
We refer to \cite{groeger87} for a similar analysis of \eqref{nlPoiss}.
	\end{rem}
Theorem \ref{t-potentialto} is crucial for our result on well-posedness, but it also provides an adequate starting
 point for an highly effective numerical solution of the nonlinear Poisson equation. We discuss this point in some
 detail: \\
Given any $k_1 \in \mathbb{R}$, e.g.\ $k_1 = 0$, with the choice of $k_2 = \mathcal{F}_2^{-1} (\mathcal{F}_1 (k_1))$,
 the pair $k = (k_1, k_2)$ is such that $\mathcal{S}(k) = 0$. Set $K_0 = \max(\vert k_1 \vert, \vert k_2 \vert)$ 
and note that $K_0=0$ is admissible if $\mathcal{F}_1 = \mathcal{F}_2$, cf.\ the examples in Subsection \ref{relations}. 
 Then by Theorem \ref{t-potentialto} ii), for all $\omega = (\omega_1, \omega_2)$ with $\Vert \omega \Vert_\infty \leq M$,
 the set of solutions $\tilde{\varphi} = \mathcal{S}(\omega)$ is bounded via 
\begin{equation*}
	\Vert \mathcal{S}(\omega) \Vert_{L^\infty}=	\Vert \mathcal{S}(\omega)- \mathcal{S}(k)\Vert_{L^\infty}
 \leq \Vert \omega - k \Vert_{\mathbf L^\infty} \leq M + K_0. 
	\end{equation*}
We use this information in the following way:
Let $K = M + K_0$, consider the function $\varpi$ with
\[
\varpi(s)= \begin{cases} K, \quad \text{if} \quad s \ge K \\
s \quad \text{if} \quad s \in [-K,K] \\
-K \quad \text{if} \quad s \le - K,
\end{cases}
\]	
and denote the induced Nemytskii operator also by $\varpi$. Then $\tilde \varphi$ is a solution of \eqref{nlPoiss} 
if and only if it satisfies the 
equation
\begin{equation} \label{e-nonliPioss}
P\tilde \varphi   -\mathcal{F}_1(\omega_1 - \varpi(\tilde{\varphi}) )
	     +     \mathcal{F}_2(\omega_2 +\varpi(\tilde{\varphi}))=0.
\end{equation}
With this cut-off in the equation, it is straightforward to check that the associated operator
\[
\mathcal P_\omega:W^{1,2}_D \ni \tilde \psi \mapsto P\tilde \psi   -\mathcal{F}_1(\omega_1 - \varpi(\tilde{\psi}) )
	     +     \mathcal{F}_2(\omega_2 +\varpi(\tilde{\psi})) \in W^{-1,2}_D
\]
is well-defined, Lipschitzian and strongly monotonous with a monotonicity constant not smaller than the one for $P:W^{1,2}_D \to W^{-1,2}_D$ . The combination of monotonicity and Lipschitz continuity in a Hilbert space setting then provides a standard, highly efficient solution algorithm for \eqref{e-nonliPioss}, based on a contraction principle, see in particular \cite[Ch. III.3.2]{ggz}.\\
Finally, a last point is interesting: due to the cut-off, these considerations do not depend on the asymptotics of
the distribution functions $\mathcal{F}_k$
at $\infty$.
\subsection[Quasilinear evolution of quasi Fermi levels ]{Quasilinear evolution of quasi-Fermi levels}
\label{sec:der}

In this subsection, we derive a quasilinear abstract Cauchy problem of type \eqref{ACP} 
that models the van-Roosbroeck system \eqref{vanRoos}. It is the basis of our analysis and of a functional analytic setting in which both gradient recombination and interfacial jump conditions can be realized, cf.\ the discussion at the beginning of this section. In particular, the smoothing through the Poisson equation \eqref{Poisson-eq} for the electrostatic potential can be fully exploited in this setting. We first give a pointwise reformulation of the bulk equations in \eqref{vanRoos} in terms of the evolution of the quasi Fermi levels $\Phi_k$ in \eqref{formalBulk} and then derive a suitable weak formulation in the space $\mathbf{X}$. 
%
With the definition \eqref{e-relation} of the quasi Fermi levels we have 
	\begin{equation}\label{phi-evo-1}
		\dot{\Phi}_k = \frac{1}{\mathcal{F}_k'(\chempot_k)} \dot{u}_k - (-1)^k \dot{\varphi}.
		\end{equation}
When recalling the split $\varphi = \tilde{\varphi} + \varphi_d$ from \eqref{splitPot}, and differentiating \eqref{tildebyu}
(formally) with respect to time, we get
 \begin{equation}\label{e-diffpoiss}
 	\dot{\varphi} = \dot{\varphi_d} + P^{-1}(\dot{u}_1 - \dot{u}_2).
 	\end{equation}
According to the defintion of the current densities \eqref{eq:curr-dens}, we get 
		\begin{equation}\label{flux-1}
		\frac{1}{\mathcal{F}_k'(\chempot_k)} \mathrm{div}\,j_k = \mathrm{div}( \frac{\mathcal{F}_k}{\mathcal{F}_k'}
(\chempot_k) \mu_k \nabla \Phi_k ) - \nabla (\frac{\mathcal{F}_k}{\mathcal{F}_k'}(\chempot_k) - \chempot_k  ) \cdot \mu_k \nabla \Phi_k.		
			\end{equation}
Combining \eqref{phi-evo-1}, \eqref{e-diffpoiss} and \eqref{flux-1} with the bulk equations in \eqref{CuCo-eq}, we obtain the equations 
\begin{equation}\label{formalBulk}
	\begin{array}{rcl}
	\left(
	\begin{array}{c}
	\dot{\Phi}_1\\
	\dot{\Phi}_2
	\end{array}
	\right)  
	\!\!\!& - &\!\!\!
	\left(
	\begin{array}{cc}
	1 +P^{-1} \mathcal{F}_1'(\chempot_1 )  & 
	-P^{-1} \mathcal{F}_2'(\chempot_2 )\\
	-P^{-1} \mathcal{F}_1'(\chempot_1) & 
	1 + P^{-1} \mathcal{F}_2'(\chempot_2 )
	\end{array}
	\right)
	\left(
	\begin{array}{c}
	\mathrm{div}( \frac{\mathcal{F}_1}{\mathcal{F}_1'}(\chempot_1) \mu_1 \nabla \Phi_1 )\\
	\mathrm{div}( \frac{\mathcal{F}_2}{\mathcal{F}_2'}(\chempot_2) \mu_2 \nabla \Phi_2 )
	\end{array}
	\right) 
	\\
	\!\!\!& = &\!\!\! 
-  \left(
	\begin{array}{cc}
	1 +P^{-1} \mathcal{F}_1'(\chempot_1 )  & 
	-P^{-1} \mathcal{F}_2'(\chempot_2 )\\
	-P^{-1} \mathcal{F}_1'(\chempot_1) & 
	1 + P^{-1} \mathcal{F}_2'(\chempot_2 )
	\end{array}
	\right)	
	\left(
	\begin{array}{c}
	\nabla (\frac{\mathcal{F}_1}{\mathcal{F}_1'}(\chempot_1) - \chempot_1  ) \cdot \mu_1 \nabla \Phi_1\\
	\nabla (\frac{\mathcal{F}_2}{\mathcal{F}_2'}(\chempot_2) - \chempot_2  ) \cdot \mu_2 \nabla \Phi_2		
	\end{array}
	\right) \\
\!\!\!& &\!\!\! + 	\left(
	\begin{array}{c}
	\frac{1}{\mathcal{F}_1'(\chempot_1)}   \\
	\frac{1}{\mathcal{F}_2'(\chempot_2)} 	
	\end{array}
	\right) r^\Omega
+ 	\left(
	\begin{array}{c}
	+ 1  \\
	-1 
	\end{array}
	\right) \dot{\varphi}_d. 
	\end{array}
	\end{equation} 
in $J\times \Omega$. \\
To incorporate the boundary and interface conditions in \eqref{CuCo-eq}, we use the split 
\begin{equation*}
	\Phi_k = \Phi_k^d + \phi_k,
	\end{equation*}
cf.\ Subsection \ref{sec:diribv}. 
We can now consider the densities $u$ in \eqref{vanRoos} as functions of $\phi$ via 
\begin{equation}\label{ubyphi}
	u_k = \mathcal{F}_k (\chempot_k),  \; with \;\chempot_k = 
\phi_k + \Phi_k^d + (-1)^k \varphi_d + (-1)^k \mathcal{S}(\phi + \Phi^d + 
\check{\varphi}^d),
	\end{equation}
	where $\mathcal{S}$ taken from \eqref{defS} is the solution operator of the nonlinear Poisson problem 
\eqref{nlPoiss} and with the notation
	\begin{equation*}
		\check {\varphi}^d = 
 \left(
		\begin{array}{c}
		+1  \\
		-1 
		\end{array}
		\right)\varphi_d .
		\end{equation*}
In the following, considering $\varphi_d$  and $\Phi^d$ as fixed, for $\phi \in \mathbf{W}^{1,q}_D$, we thus define
 \begin{equation*}
 \tilde{\mathcal{F}}_k (t,\phi) = \mathcal{F}_k (\chempot_k(t))
 \end{equation*} 
 with the right-hand-side as in \eqref{ubyphi} and, correspondingly, $\tilde{\mathcal{F}}_k' (t,\phi) = \mathcal{F}_k' (\chempot_k(t) )$, and 
 \begin{equation*}
	 \eta_k(t,\phi) = \frac{\tilde{\mathcal{F}}_k (t,\phi_k)}{\tilde{\mathcal{F}}_k' (t,\phi_k)}.
	 \end{equation*}
		 As an additional shorthand, we write 
		 \begin{equation*}
		 	\left(
		 	\begin{array}{cc}
		 	1 +P^{-1} \mathcal{F}_1'(\chempot_1 )  & 
		 	-P^{-1} \mathcal{F}_2'(\chempot_2 )\\
		 	-P^{-1} \mathcal{F}_1'(\chempot_1) & 
		 	1 + P^{-1} \mathcal{F}_2'(\chempot_2 )
		 	\end{array}
		 	\right) = \mathrm{Id} + P^{-1}[\tilde{\mathcal{F}}'(t,\phi)],
			 \end{equation*}
			 for the matrix operators in \eqref{formalBulk}.\\
We can now define the abstract evolution problem \eqref{vanRoos}, in a functional analytic setting in which Neumann boundary and interfacial recombination terms appear on the right-hand-sides,
\begin{equation}\label{abstract}
	\dot{\phi}(t) + \mathcal{A}(t,\phi(t))\phi(t) = \mathcal{R}(t,\phi(t)) \in \mathbf{X} \quad \text{ for a.a. }t \in J. 
	\end{equation}
The operators $\mathcal{A} \colon \overline{J}\times \mathbf{W}^{1,q} \to \mathcal{L}(\mathcal{D}_{\mu_k},\mathbf{X})$ and $ \mathcal{R} = \mathcal{R}_{\mathrm{flux}} + \mathcal{R}_{\mathrm{rec}} + \mathcal{R}_{\mathrm{data}}$ are given by the elliptic part
	\begin{equation}\label{defAcal}
		\mathcal{A}(t,v)\phi =  (\mathrm{Id} + P^{-1}[\tilde{\mathcal{F}}'(t,v)])
	\left(
	\begin{array}{cc}
A_{\eta_1(t,v) \mu_1}   & 
	0\\
	0 & 
	A_{\eta_2(t,v) \mu_2}
	\end{array}
	\right)
	\phi,
		\end{equation} 
and the lower-order flux term $\mathcal{R}_{\mathrm{flux}} \colon J \times \mathbf{W}^{1,q}_D \to \mathbf{L}^{q/2}$ with
\begin{equation}\label{defRflux}
	\mathcal{R}_{\mathrm{flux}} (t,v) = (\mathrm{Id} + P^{-1}[\tilde{\mathcal{F}}'(t,v)])
	\left(
	\begin{array}{c}
	\nabla (\eta_1(t,v) - v_1  ) \cdot \mu_1 \nabla v_1 \\
		\nabla (\eta_2(t,v) - v_2  ) \cdot \mu_2 \nabla v_2 
	\end{array}
	\right).
	\end{equation}
In order to define the recombination term $\mathcal{R}_{\mathrm{rec}} \colon J \times 
\mathbf{W}^{1,q}_D \to \mathbf{X}$ with 

	 \begin{equation}\label{e-defRrec}
		 \mathcal{R}_{\mathrm{rec}}(t,v) = \left(
	\begin{array}{c}
	\frac{1}{\tilde{\mathcal{F}}_1'(t,v)} \\
		\frac{1}{\tilde{\mathcal{F}}_2'(t,v)}
	\end{array} \right) \Big( \tilde{r}^\Omega(t,v) + \tilde{r}^\Gamma(t,v) + \tilde{r}^\Pi(t,v) \Big),
		 \end{equation}
		 we set $\tilde{r}^E (t,\phi) = r^E(u,\varphi,\Phi)$ for $E\in \{\Omega, \Gamma, \Pi\}$ with $u$ and $\varphi$ as in \eqref{ubyphi}.
 We consider $\mathcal{R}_{\mathrm{rec}}(t,v)$ as an element of $\mathbf{X}$ by the embeddings in Lemma \ref{t-propX}. 
	 The part of the right-hand-side in \eqref{abstract} modeling inhomogeneous data is given by $\mathcal{R}_{\mathrm{inh}}
\colon J \times \mathbf{W}^{1,q}_D \to \mathbf{X}$ with \begin{align}\label{defRinh}
	 \mathcal{R}_{\mathrm{inh}}(t,v) & = -\dot{\Phi}^d(t) + 
	 (\mathrm{Id} + P^{-1}[\tilde{\mathcal{F}}'(t,v)])
	\left(
	\begin{array}{cc}
\widehat{A}_{\eta_1(t,v) \mu_1}   & 
	0\\
	0 & 
	\widehat{A}_{\eta_2(t,v) \mu_2}
	\end{array}
	\right)
	\Phi^d \\
	& + (\mathrm{Id} + P^{-1}[\tilde{\mathcal{F}}'(t,v)])
	\left(
	\begin{array}{c}
	\nabla (\eta_1(t,v) - v_1  ) \cdot \mu_1 \nabla \Phi^d_1 \\
		\nabla (\eta_2(t,v) - v_2  ) \cdot \mu_2 \nabla \Phi^d_2 
	\end{array}
	\right)
	+ \dot{\check{\varphi}}^d.
	 \end{align}
	 
The operators $\mathcal{A}$ and $\mathcal{R}$ are analyzed further in Subsection \ref{sec:proof} below where it is shown that they 
adapt to the functional analytic setting in $\mathbf{X}$ and that they are locally Lipschitz in $v$ uniformly with respect to time. 

\begin{rem} \label{r-Boltz}
In case of Boltzmann statistics, $\mathcal F_k = \exp$ one has $\eta_k = 1$, and the main part of the parabolic operator
 in \eqref{defAcal}
simplifies to a linear one. This shows why the analysis of van Roosbroeck's system is then much easier, compare 
\cite{gajewski:groeger89}.
\end{rem}

\section[Main Result]{Main Result}
\label{sec:main}

%
In this section, we state the main result on well-posedness and regularity of solutions of the van Roosbroeck system. In the 
proof, we use the concept of maximal parabolic regularity and its application to quasilinear problems. Known preliminary 
results are stated in Subsection \ref{MaxReg}. In Subsection \ref{sec:proof}, we show that due to our preliminary considerations
 in Sections \ref{sec:rigorous} and \ref{sec:nl-reform}, the abstract theory can be applied to \eqref{abstract}. 
In Subsection \ref{r-Pi}, we discuss further implications and related topics. 
\begin{thm}\label{mr}
	Under Assumptions \ref{spatial}, \ref{assu:distri}, \ref{assu:iso}, \ref{assu:recomb} and \ref{assu:poi-diri-bv}, let $3<q<4$ as in Assumption\ \ref{assu:iso} and let $s > \frac {2q}{q-3}$. 
	\begin{itemize}
		\item {\bf Local well-posedness:}
Suppose 
\begin{equation*} 
\phi_0 = \Phi_0 - \Phi^d(0) \in (\mathbf X,\mathcal D_\mu)_{1-\frac {1}{s},s} = \mathbf{Y}_{s,q}.
\end{equation*}
Then there is a maximal time interval $J^*=]0,T^*[$ of existence ($0<T^*\leq T$)
and a unique solution 
\begin{equation*}
	\phi \in L^s (J^*; \mathcal{D}_\mu) \cap W^{1,s}(J^*;\mathbf{X})\cap C(\overline{J^*};\mathbf{Y}_{s,q})
\hookrightarrow C(\overline{J^*};\mathbf W^{1,q})
	\end{equation*}
	of \eqref{abstract} that depends continuously on the data and initial value in the respective norms. 
	\item The {\bf electron and hole densities} and the  {\bf chemical and electrostatic potentials} associated to the solution $\phi$
 satisfy 
	\begin{equation*}
		u_k, \chi_k, \varphi \in C(J^*;W^{1,q})\hookrightarrow C(J^*;C^\beta), \qquad u > 0,
		\end{equation*} 
			for some $\beta > 0$. 
\item {\bf Regularity in time:} If the data $d,\Phi^D, \varphi_D$ and $\varphi_\Gamma$ are such that there is a $\gamma > 0$ with $\mathcal{R}_{\mathrm{inh}}(\cdot,v) \in C^\gamma(J;\mathbf{X})$ for every $v\in \mathbf{W}^{1,q}$, then 
\begin{equation*}
	\phi \in C^\gamma(J^*; \mathcal{D}_\mu) \cap C^{1+\gamma}(J^*; \mathbf{X}).
	\end{equation*}
	\end{itemize}
	\end{thm}

\subsection[Maximal parabolic regularity]{Maximal parabolic regularity}
\label{MaxReg}
%
			The proof of Theorem \ref{mr} rests
			on the notion of \emph{maximal parabolic regularity} for a suitable linearization of the problem, which
			we recall here:
			\begin{defn} \label{d-maxreg}
			Let $1 < s < \infty$, let $Z$ be a Banach space and let $J := \left]0, T \right[$ 
			be a bounded interval. Assume that $B$ is a closed
			operator in $Z$ with dense domain $ \mathfrak D$,  equipped with the graph norm.
			We say that $B$ satisfies \emph {maximal parabolic $L^s$-regularity in $Z$}, if for any
			$f\in L^s(J; Z)$ there exists a unique function $v \in W^{1,s}(J; Z) \cap
			L^s(J; \mathfrak D)$ satisfying $ v(0) = 0$ and
			\begin{equation}\label{mpr} \dot{v} + Bv = f   \quad \text{ holds a.e. on} \; J.
			\end{equation}
                      
					\end{defn}
			\begin{rem} \label{r-ebed}
			\renewcommand{\labelenumi}{\roman{enumi})}
			\begin{enumerate}

			\item The property of maximal parabolic regularity of
			        an operator $B$ is independent of $s \in ]1, \infty [$ and
			        the choice of a bounded interval $J$, cf. \cite[Thm. 7.1/Cor. 5.4]{dore1}.
			
			\item Observe that (cf. \cite[Ch.~4.10]{amannbuch})
			        \begin{equation*}
			          W^{1,s}(J; Z) \cap L^s(J; \mathfrak D) \hookrightarrow
			                C(\overline{J}; (Z, \mathfrak D)_{1-\frac{1}{s},s}).
			        \end{equation*}
					In particular, $(Z, \mathfrak D)_{1-\frac{1}{s},s}$, is the appropriate space of 
                        initial values for \eqref{mpr}.
			\item
			If $\theta \in ]0,1-\frac {1}{s}[$, then
			\begin{equation*} 
			 W^{1,s}(J; Z) \cap L^s(J; \mathfrak D) \hookrightarrow
			                C^\beta({J}; (Z, \mathfrak D)_{\theta,1})
			\end{equation*}
			with $\beta:=1-\frac {1}{s} -\theta$, cf. \cite[Thm. 3]{amann}.
			\item 
			If $B$ satisfies maximal parabolic regularity on a Banach space $Z$, and $B_0 $ is
			relatively bounded with a sufficiently small relative bound, then $B+B_0$ also
			satisfies maximal parabolic regularity on $Z$, cf. \cite[Prop. 1.3]{a/c/f/p} or \cite[Prop. 1.5]{pruess}. 
			\item
			If $B$ satisfies maximal parabolic regularity on the complex Banach space $Z$, then $-B$ is a generator of
			an analytic semigroup on $Z$. \cite[Ch. 4]{dore1}.
			\item
			If $B_1, B_2$ satisfy maximal parabolic regularity on $Z$, then $\begin{pmatrix} B_1 & 0 \\ 
			0 & B_2 \end{pmatrix}$ satisfies maximal parabolic regularity on $\mathbf{Z}=Z \oplus Z$. 

			\end{enumerate}
			\end{rem}
We first show that the
second order divergence operators $A_\rho$ occurring in \eqref{abstract} satisfy maximal parabolic regularity:
\begin{prop} \label{t-zwischen} 
Let $\rho$ be an elliptic coefficient function on $\Omega$, and assume 
$q \in [2,\infty[$.
\renewcommand{\labelenumi}{\roman{enumi})}
\begin{enumerate}
\item
Then the operator $\underline A_\rho$ satisfies  maximal parabolic regularity in $\underline W^{-1,q}_D$ and on $\underline L^q$.
\item
If $\theta \in ]0,1[$, then it also satisfies maximal parabolic regularity in $[\underline L^q,\underline W^{-1,q}_D]_\theta.$
\end{enumerate}
\end{prop}
\begin{proof}
Maximal parabolic parabolic regularity in $\underline L^q$ is obtained under our supposed geometric conditions, if one uses the  upper Gaussian estimates for the semigroup kernel from  \cite{elst} and then applies \cite{HiebPruess}, compare also \cite{coulhon}. For the case $\underline W^{-1,q}_D$, see \cite[Ch. 11]{ausch}.
ii) follows from i) and the following fact, proved in \cite[Lemma 5.3]{hal/reh}: if the (complex) Banach space $Z_1$ 
embeds into the (complex) Banach space $Z_2$ and the operators $A:dom_{Z_2}(A) \to Z_2$ and $A|_{Z_1}$ satisfy maximal parabolic
regularity on $Z_1$ and $Z_2$, respectively, then $A$ also satisfies maximal parabolic regularity on every
complex interpolation space $[Z_1,Z_2]_{\theta}$. Compare also \cite[Thm. 5.19]{hal/rehBedl}.
\end{proof}
\begin{cor} \label{t-zwischenreell} 
Let $\rho$ be an elliptic coefficient function on $\Omega$, and assume  
$q \in [2,\infty[$.
If $\theta \in ]0,1[$, then $A_\rho$ also satisfies maximal parabolic regularity in $[L^q,W^{-1,q}_D]_\theta=:Z.$
\end{cor}
\begin{proof}
We assume 
$q$ as fixed and define $Z:=[L^q,W^{-1,q}_D]_\theta$,  $\underline Z:=[\underline L^q,\underline W^{-1,q}_D]_\theta$.
Let $f \in L^s(J;Z)$. 
We identify an element $\mathrm z \in Z$ with an element $\underline {\mathrm z}  \in \underline Z$ by setting,
\begin{equation*} 
\langle \underline  {\mathrm z}|\psi \rangle_{\underline Z}:= \langle \mathrm z|\psi_1 \rangle_Z - i \langle \mathrm z| \psi_2 \rangle_Z , 
\qquad \psi = \psi_1 + i \psi_2 \in \underline Z^*=[\underline L^{q'},\underline W^{1,q'}_D]_\theta.
\end{equation*}
Identifying $f$ in this spirit with a function $g \in  L^s(J;\underline Z)$, we are looking for a solution $v$
of the equation
\begin{equation} \label{e-paar}
\dot{v} +\underline A_\rho v =g, \quad v(0)=0,
\end{equation}
According to the maximal parabolic regularity of $\underline A_\rho$ 
on $\underline Z$, the (unique) solution of \eqref{e-paar} exists and belongs to the space 
$L^s(J;dom_{\underline Z}(\underline A_\rho) \cap W^{1,s}(J;\underline Z)$. But, according to \cite[Ch. III1.3
Prop. 1.3.1]{amannbuch}, the solution of \eqref{e-paar} is given by the variation of constants formula
\[
v(t) = \int_0^t e^{-(t-s)\underline A_\rho}g(s) \, ds.
\]
Here one observes that the semigroup operators $e^{-(t-s)\underline A_\rho}$ transform elements of $Z$ into real elements
of $dom_{\underline Z}(\underline A_\rho)$ since the resolvent also has this behaviour. Thus, $ v \in L^s(J;dom_{Z}(A_\rho))$.
But $\underline A_\rho$ acts on $dom_{Z}(A_\rho)$ as $A_\rho$; so the equation \eqref{e-paar} shows that 
$\dot{v} \in L^s(J;Z)$, proving the assertion.
\end{proof}
%
%
%
%
%
			The proof of Theorem \ref{mr} rests on the maximal parabolic regularity of the 
linearization of \eqref{abstract} and a Banach fixed point argument, which is encoded 
in the following Proposition. 
			
\begin{prop}[\cite{pruess}] 
\label{p-pruess} 
			Suppose that $B$ is a closed operator on a Banach space $Z$ with dense
			domain $\mathfrak D$, which satisfies maximal parabolic regularity on $Z$.
			Suppose further $v_0 \in (Z, \mathfrak D)_{1-\frac {1}{s},s}$ and
			$\mathcal{B}: \bar J \times (Z, \mathfrak D)_{1-\frac {1}{s},s} \to \mathcal{L}
			({\mathfrak D}, Z)$ to be continuous with $B = \mathcal{B} (0,v_0)$. Let,
			in addition, $\mathcal R : J \times (Z, \mathfrak D)_{1-\frac {1}{s},s} \to Z$ be a
			Carath\'eodory map and assume the following Lipschitz conditions on
			$\mathcal B$ and $\mathcal R$:
			\begin{itemize}
			\item[$({\bf LA})$] For every $M > 0$ there exists a constant $C_M > 0$, such
			        that for all $t \in J$
			        \[
			          \| \mathcal B(t,w) - \mathcal B(t, \tilde{w}) \|_{\mathcal
			                L( \mathfrak D,Z)} \le C_M \; \| w - \tilde {w}\|_{(Z, \mathfrak D)_{1-\frac {1}{s},s}} ,
			\]
			if $\|w\|_{(Z, \mathfrak D)_{1-\frac {1}{s},s}},\; \| \tilde{w}
			                \|_{(Z, \mathfrak D)_{1-\frac {1}{s},s}} \le M$.
			\item[$({\bf LB})$] $\mathcal R(\cdot, 0) \in L^s(J;Z)$, and for each $M>0$ there is a
			        function $h_M \in L^s(J)$, such that
			        \[
			          \| \mathcal R(t,w) - \mathcal R(t,\tilde{w})\|_Z \le h_M(t) \; \| w
					- \tilde{w} \|_{(Z, \mathfrak D)_{1-\frac {1}{s},s}}
			        \]
			        holds for a.a. $t \in J$, if
			        $\|w\|_{(Z, \mathfrak D)_{1-\frac {1}{s},s}},
			        \|\tilde{w}\|_{(Z, \mathfrak D)_{1-\frac {1}{s},s}} \leq M$.
			\end{itemize}
			Then there exists $T^* \in J \cup \{T\}$, such that the equation
			\begin{equation*} 
			  \left\{ \begin{aligned}
			        \dot{v}(t) + \mathcal B \bigl(t, v(t) \bigr) v(t) &= \mathcal R(t,v(t)),\quad
					\text{a.e. } t \in J, \\
			        v(0) &= v_0.
			        \end{aligned} \right.
			\end{equation*}
			 admits a unique solution $v$ satisfying
			\begin{equation*}
			 v \in W^{1,s}(0,T^*;Z) \cap L^s(0,T^*;\mathfrak  D).
			\end{equation*}
The solution depends continuously on the initial condition in $(Z, \mathfrak D)_{1-\frac {1}{s},s}$ and the maximal time of existence $T^*$ is characterized by either $T^*=T$ or 
		\begin{equation*}
		\Vert v(t) \Vert_{(Z,\mathfrak D)_{1-\frac {1}{s},s}} \to +\infty \quad \text{as } t \to T^*. 	
			\end{equation*}
			\end{prop}

%
\subsection{Proof of  Theorem \ref{mr}}\label{sec:proof}
%
 As a next step, we prove the first part of Theorem \ref{mr}.
The proof is an application of Proposition \ref{p-pruess}. Some preliminary observations:
\begin{lem} \label{t-multipl}
	Recall $X=[L^q,H^{-1,q}_D]_{\frac {3}{q}}$. Assume that $\rho$ is an elliptic coefficient function, such that 
\[
A_\rho:W^{1,q}_D \to W^{-1,q}_D
\]
is a topological isomorphism. 
\renewcommand{\labelenumi}{\roman{enumi})}
			\begin{enumerate}
\item
Then the (linear) mapping
\begin{equation*}
W^{1,q} \ni \eta \mapsto A_{\eta \rho} \in \mathcal L(dom_{X}(A_\rho);X)
\end{equation*}
is well-defined and continuous with norm $c \|\eta\|_{W^{1,q}_D}$, where the constant $c$
depends only on $\Omega$, $D$ and $\rho$. In particular, $dom_X(A_\rho) \subseteq  dom_X(A_{\eta \rho})$.
\item
Assume that the function $\eta \in W^{1,q}$ admits a strictly positive lower bound.
Then $dom_X(A_{\eta \rho})= dom_X(A_{\rho})$ and the corresponding graph norms are
equivalent.
\end{enumerate}
\end{lem}
\begin{proof}
i) in \cite[pp. 1384/1385]{hal/reh}, it is proved that
\begin{equation} \label{e-estJDE}
\|\underline A_{\eta \rho} \psi \|_{\underline X} \le   c \|\eta\|_{W^{1,q}_D}\|\psi
\|_{dom_{\underline X}(\underline A_\rho)}, \quad \psi \in dom_{\underline X}(\underline A_\rho),
\end{equation}
for some constant $c >0$.
The proof immediately carries over to the case of real spaces.\\
ii) The properties of $\eta $ guarantee that also 
\[
A_{\eta \rho}:W^{1,q}_D \to W^{-1,q}_D
\]
is a toplogical isomorphism, cf.\ Remark \ref{r-kommentar}. Thus, 
the result is obtained by replacing 
 $\rho$ by $\eta \rho$ in i) and, afterwards, $\eta$ by $\frac {1}{\eta}$.
\end{proof}

\begin{lem} \label{l-maxparpert}
Assume that $f_1,f_2, \eta_1, \eta_2  \in W^{1,q}$ and suppose that $\eta_1, \eta_2$
are bounded functions with strictly positive lower bounds. 
			\renewcommand{\labelenumi}{\roman{enumi})}
			\begin{enumerate}
\item
Then
\begin{equation} \label{e-maxparpert}
dom_{\mathbf X}\left(
(\mathrm{Id} + P^{-1}[f])
\left(
\begin{array}{cc}
A_{\eta_1 \mu_1} & 0 \\
0 &  A_{\eta_2 \mu_2}
\end{array}
\right) \right) =
dom_{\mathbf X} \left(
\begin{array}{cc}
A_{ \mu_1} & 0\\
0 &  A_{\mu_2}
\end{array}
\right) ,
\end{equation}
\item
and, moreover, the operator
\[
(\mathrm{Id} + P^{-1}[f])
\left(
\begin{array}{cc}
A_{\eta_1 \mu_1} & 0 \\
0 &  A_{\eta_2 \mu_2}
\end{array}
\right)
\]
has  maximal parabolic regularity on $\mathbf X$.
\end{enumerate}
\end{lem}
\begin{proof}
By Lemma \ref{t-multipl}, one has 
\[
dom_{\mathbf X} \left(
\begin{array}{cc}
A_{ \mu_1} & 0\\
0 &  A_{\mu_2}
\end{array}
\right) =
dom_{\mathbf X} \left(
\begin{array}{cc}
A_{ \eta_1 \mu_1} & 0\\
0 &  A_{\eta \mu_2}
\end{array}
\right), 
\]
and it is clear that the functions $f_k$ act as continuous
 multiplication operators on $X$. Moreover, $P^{-1}:X \to X$ is compact. Hence, the operator
\begin{equation} \label{e-relcomp}
P^{-1} [f]
\left(
\begin{array}{cc}
A_{\eta_1 \mu_1} & 0\\
0 &  A_{\eta_2 \mu_2}
\end{array}
\right)
\end{equation}
is relatively compact with respect to $\left(
\begin{array}{cc}
A_{\eta_1 \mu_1} & 0\\
0 & A_{\eta_2 \mu_2}
\end{array}
\right)$. This implies \eqref{e-maxparpert}, cf. \cite[Ch. IV.1.3]{kato}.\\
ii) The operator $\left(
\begin{array}{cc}
A_{\eta_1 \mu_1} & 0\\
0 &  A_{\eta_2 \mu_2}
\end{array}
\right)$ satisfies maximal parabolic regularity on $\mathbf X$, cf. Proposition \ref{t-zwischen}  and
Remark \ref{r-ebed}. As established in i), \eqref{e-relcomp} is relatively compact with respect to
 $\left(
\begin{array}{cc}
A_{\eta_1 \mu_1} & 0\\
0 &  A_{\eta_2 \mu_2}
\end{array}
\right)$.  Using the reflexivity of $\mathbf X$,  this implies that \eqref{e-relcomp} is relatively
 bounded with respect to   $\left(
\begin{array}{cc}
A_{\eta_1 \mu_1} & 0\\
0 &  A_{\eta_2 \mu_2}
\end{array}
\right)$, and the relative bound may be taken arbitrarily small, cf. \cite{bind}. Having this at hand,
a suitable perturbation theorem applies, cf.\ Remark \ref{r-ebed}.
\end{proof}

\begin{cor}\label{ttrace}
	Let $s,q$ and $\mathbf{Y}_{s,q}$ as in Theorem \ref{mr}. Then, for every function $v \in L^s(J;\mathcal{D}_{\mu_k}) \cap W^{1,s}(J;\mathbf{X})$, by Remark \ref{r-ebed}, we have $v \in C(\overline{J}; \mathbf{Y}_{s,q})$. Moreover, by Corollary \ref{c-interint}, 
	\begin{equation*} 
	(\mathbf X,\mathcal D_\mu)_{1-\frac {1}{s},s} = \mathbf{Y}_{s,q}
	\hookrightarrow (\mathbf X,\mathcal D_\mu)_{1-\frac {1}{s},\infty} \hookrightarrow \mathbf W^{1,q}_D.
	\end{equation*}
	\end{cor}
Now we are in the position to show Theorem \ref{mr} by applying Proposition \ref{p-pruess}:\\
From Lemmas \ref{t-multipl} and \ref{l-maxparpert} and Corollary \ref{ttrace}, it follows that the operator $\mathcal{A}$ in \eqref{defAcal} is well-defined. In particular, for given $v\in \mathbf{Y}_{s,q} \hookrightarrow \mathbf{W}^{1,q}_D$, $\eta_k(t,v) \in W^{1,q}$ is bounded from above and below by positive constants. Moreover, by Lemma \ref{l-maxparpert},  $\mathcal A(0,\phi(0))$
satisfies maximal parabolic regularity in $\mathbf X$. 
Secondly, using Lemma \ref{t-multipl}, it is not hard to see that $\mathbf{(LA)}$ in Proposition \ref{p-pruess}
 also holds, with the following example of an explicit estimate:
\begin{align*}
	 & \Vert \mathcal{A}(t,v) - \mathcal{A}(t,w) \Vert_{\mathcal{L}(\mathcal{D}_{\mu_k},\mathbf{X})} \\
	& \leq C \max_k \left[ \Vert P^{-1} \Vert_{\mathcal{L}(L^\infty,X)} L_{\tilde{\mathcal{F}}_k'(t)} \Vert v_k - w_k 
\Vert_{L^\infty} C_{\eta_k(t)} \Vert v_k \Vert_{W^{1,q}_D} \right]\\
	&  \quad + C \max_k \left[ (1+ \Vert P^{-1} \Vert_{\mathcal{L}(L^\infty,X)} C_{\tilde{\mathcal{F}}_k'(t)} \Vert w_k 
\Vert_{L^\infty}) L_{\eta_k(t)} \Vert v_k - w_k \Vert_{W^{1,q}_D})  \right] \\
	& \leq C \Vert v - w \Vert_{\mathbf{Y}_{s,q}},
	\end{align*}
	where $L_f$ is a local Lipschitz constant and $C_f$ is a local bound on the real-valued function $f$ and 
$C>0$ is a generic constant that, in particular, contains embedding constants and the constant in \eqref{e-estJDE}. 
Here, we implicitly used the Lipschitz property of $\mathcal{S}: \mathbf{L}^{\infty} \to W_D^{1,q}$, Thm. \ref{t-potentialto}
to have Lipschitz dependence of the coefficient functions $\tilde{\mathcal{F}}_k(t,\cdot), \eta_k(t,\cdot)$ of $v,w$. \\
For the right-hand-side $\mathcal R_{\mathrm{flux}}$ in \eqref{defRflux}, we analogously obtain $\mathbf{(LB)}$ in Proposition \ref{p-pruess}
 by the embedding $L^{q/2} \hookrightarrow X$ in Lemma \ref{t-propX}. \\
For the right-hand-side $\mathcal R_{\mathrm{rec}}$ in \eqref{e-defRrec}, Lipschitz-dependence follows from Assumptions \ref{assu:recomb} and 
\ref{recombGamma}, Lemma \ref{l-Lipsch} and the embeddings in Lemma \ref{t-propX}.  \\
The remaining term $\mathcal{R}_{\mathrm{inh}}$ in \eqref{defRinh} is treated analogously, taking into account Assumptions
 \ref{assu:poi-diri-bv} on the data.  
This proves the first part of Theorem\ \ref{mr}. The second part of Theorem \ref{mr} follows directly from the relations
 \ref{ubyphi} and \ref{tildebyu} of $\phi$ and $u,\varphi, \chi$, together with Thm. \ref{t-potentialto}. Spatial H\"older
 regularity is a consequence of the standard embedding $W^{1,q} \hookrightarrow C^\beta$ for $q>3$. 
The third part of Theorem \ref{mr} is a direct consequence of well-known theory for nonautonomous parabolic problems,
 cf.\ \cite[Thm. 4.3]{prato} and compare also \cite[Cor. 6.1.6]{luna}.

\subsection[Concluding remarks]{Concluding remarks}\label{r-Pi}
%
We conclude with a few remarks on direct extensions and open problems associated with the main result. \\
The equations in two spatial dimensions can be analyzed in exactly the same way, leading to an analogous result. Assumption \ref{assu:iso} that restricts the geometric setting and coefficients can then be dropped in the sense that for all bounded, measurable and elliptic coefficient functions, there exists a suitable exponent $q>2$, cf. \cite{jonsson}. \\
	Note that if $r^\Pi \neq 0$, the solution $\phi$ in the main result Theorem\ \ref{mr} will in general not be twice (weakly) differentiable and the regularity in Theorem\ \ref{mr} is optimal in this sense. If $r^\Pi = 0$ and the setting is smooth, e.g.\ $D=\partial \Omega$, the material coefficients $\mu_k, \varepsilon, \varepsilon_\Gamma$ and the boundary and initial data are smooth, then it is straightforward to obtain higher spatial regularity and a strong solution of \eqref{vanRoos} from our method by using elliptic regularity in $L^p$ and a boot-strap argument.\\
The Poisson equation \eqref{Poisson-eq} for the electrostatic potential is sometimes considered on a larger domain than the current-continuity equation \eqref{CuCo-eq}, 
 cf. \cite{vanroos2d}. This extension is also possible with our analysis. \\
Finally, it would be interesting to identify the interpolation space
$[L^q,W^{-1,q}_D]_\tau$ 
with a dual space of Bessel potentials $H^{-\tau,q}_D = \Bigl (H^{\tau,q'}_D \bigr )^*$. 
This is known for more specific geometries, i.e. if $\Omega$ is a Lipschitz domain, 
$D$ is the closure of its interior (within $\partial \Omega$), and the boundary of $D$ (within $\partial \Omega)$ is locally bi-Lipschitzian diffeomorphic to the unit interval,
see \cite{ggkr} and \cite[Ch. 5]{haller}. Under our more general Assumption \ref{spatial},
the proof seems to be a very hard task. 

\small
\noindent
{\sc K. Disser,\\
Weierstrass Institute for Applied Analysis and Stochastics,
Mohrenstr.~39, 
10117 Berlin, 
Germany}  and \\
{\sc TU Darmstadt, IRTG 1529, 
Schlossgartenstr.~7, 
64289 Darmstadt, 
Germany}  \\
{\em E-mail address}\/: {\bf disser@wias-berlin.de, kdisser@mathematik.tu-darmstadt.de} 

\mbox{}

\noindent
{\sc J. Rehberg, \\
Weierstrass Institute for Applied Analysis and Stochastics,
Mohrenstr.~39, 
10117 Berlin, 
Germany}  \\
{\em E-mail address}\/: {\bf rehberg@wias-berlin.de}

\end{document}